\documentclass[11pt]{amsart}
\usepackage[square,numbers,sort&compress]{natbib}
\usepackage{graphicx}

\usepackage[utf8]{inputenc}
\usepackage{amsmath, amsthm, amssymb}
\usepackage[all]{xy}
\usepackage{accents}
\usepackage[mathscr]{euscript}
\usepackage{braket}
\usepackage{url}
\numberwithin{equation}{section}
\newtheorem{lemma}{Lemma}[section]
\newtheorem{theorem}[lemma]{Theorem}
\newtheorem{conjecture}[lemma]{Conjecture}
\newtheorem{corollary}[lemma]{Corollary}
\newtheorem{proposition}[lemma]{Proposition}
\theoremstyle{definition}
\newtheorem{remark}[lemma]{Remark}
\newtheorem{definition}[lemma]{Definition}

\def\tilde{\widetilde}
\def\hat{\widehat}
\def\Sp{{\mathrm{Sp}}}

\def\GL{{\mathrm{GL}}}

\def\Sym{{\mathrm{Sym}}}
\DeclareMathOperator{\Ad}{Ad}
\DeclareMathOperator{\Hilb}{Hilb}

\DeclareMathOperator{\PHilb}{PHilb}
\DeclareMathOperator{\Spec}{Spec}

\DeclareMathOperator{\Frac}{Frac}

\DeclareMathOperator{\id}{\mathrm{id}}

\DeclareMathOperator{\Lie}{Lie}

\DeclareMathOperator{\Fl}{Fl}
\DeclareMathOperator{\Gr}{Gr}

\DeclareMathOperator{\diag}{\mathrm{diag}}
\DeclareMathOperator{\gr}{gr}

\newcommand{\bPic}{\overline{\mathrm{Pic}}}

\newcommand{\tM}{\widetilde{M}}
\newcommand{\cO}{\mathcal{O}}
\newcommand{\cP}{\mathcal{P}}
\newcommand{\cL}{\mathcal{L}}
\newcommand{\cF}{\mathcal{F}}
\newcommand{\cA}{\mathcal{A}}
\newcommand{\cR}{\mathcal{R}}
\newcommand{\cH}{\mathcal{H}}
\newcommand{\cM}{\mathcal{M}}
\newcommand{\cC}{\mathcal{C}}
\newcommand{\cI}{\mathcal{I}}
\newcommand{\cJ}{\mathcal{J}}
\newcommand{\cN}{\mathcal{N}}
\newcommand{\cT}{\mathcal{T}}
\newcommand{\cV}{\mathcal{V}}
\newcommand{\cK}{\mathcal{K}}

\newcommand{\fs}{\mathfrak{s}}
\newcommand{\fS}{\mathfrak{S}}

\newcommand{\hC}{\widehat{C}}

\newcommand{\fl}{\mathfrak{l}}

\newcommand{\fsl}{\mathfrak{sl}}
\newcommand{\fg}{\mathfrak{g}}
\newcommand{\ft}{\mathfrak{t}}
\newcommand{\tT}{\widetilde{T}}

\newcommand{\bI}{\mathbf{I}}
\newcommand{\bP}{\mathbf{P}}
\newcommand{\tbP}{\widetilde{\mathbf{P}}}
\newcommand{\bN}{\mathbf{N}}

\newcommand{\Sn}{\mathfrak{S}_n}

\newcommand{\C}{{\mathbb C}}

\newcommand{\A}{\mathbb{A}}
\newcommand{\tG}{\widetilde{G}}
\newcommand{\G}{\mathbb{G}}
\newcommand{\Z}{{\mathbb Z}}
\newcommand{\Q}{{\mathbb Q}}
\newcommand{\cW}{\mathcal{W}}
\newcommand{\F}{{\mathbb F}}

\newcommand{\Hom}{\mathrm{Hom}}

\renewcommand{\P}{\mathbb{P}}

\graphicspath{{./figures/}}

\begin{document}
\title[GASF and Hilbert schemes]{Generalized affine Springer theory and Hilbert schemes on planar curves}
\author{Niklas Garner}
\address{Department of Physics, University of Washington}
\author{Oscar Kivinen}
\address{Department of Mathematics, EPFL}
\begin{abstract}
We show that Hilbert schemes of planar curve singularities and their parabolic variants can be interpreted as certain generalized affine Springer fibers for $GL_n$, as defined by Goresky-Kottwitz-MacPherson. Using a generalization of affine Springer theory for Braverman-Finkelberg-Nakajima's Coulomb branch algebras, we construct a rational Cherednik algebra action on the homology of the Hilbert schemes, and compute it in examples. Along the way, we generalize to the parahoric setting the recent construction of Hilburn-Kamnitzer-Weekes, which may be of independent interest. In the spherical case, we make our computations explicit through a new general localization formula for Coulomb branches. Via results of Hogancamp-Mellit, we also show the rational Cherednik algebra acts on the HOMFLY-PT homologies of torus knots. This work was inspired in part by a construction in three-dimensional $\cN=4$ gauge theory.
\end{abstract}
\maketitle
\section{Introduction}
Let $\hC$ be the germ of a (possibly non-reduced and reducible) complex plane curve singularity and write $\hC = \Spec \C[[x,t]]/(f)$ for any nonzero $f \in \C[x][[t]]$ with $f(0,0) = 0$, whose closed points correspond to the vanishing set $\{f(x,t) = 0\}$. In this paper, we investigate a relationship between the Hilbert scheme of points on $\hC$  (plus its parabolic flag versions) and certain generalized affine Springer fibers in the sense of \cite{GKM}.

\begin{remark}
By the Weierstrass preparation theorem, there is no loss of generality in assuming that $f$ has finite degree in $x$. By the local constancy of classical affine Springer fibers \cite[Proposition 3.5.1.]{Ngo}, we may even assume $f(x,t)\in \C[x,t]$ is a polynomial of both $x$ and $t$, though its degree in terms of $t$ may be very large.
\end{remark}

The Hilbert schemes of points on singular curves have been objects of intense study due to their connections to a wide range of topics including knot theory \cite{ORS, GORS}, representation theory \cite{GORS, Nak, Kiv1, OY, EGL}, and curve counting \cite{PT, Pand}.
Affine Springer fibers, and their various generalizations, have also seen a wide range of study in combinatorics \cite{Hik}, geometry \cite{Ngo, LS}, number theory \cite{Ngo, YunPCMI}, and representation theory \cite{OY,VV}.
\subsection{Hilbert schemes and affine Springer fibers}
We now describe our approach in some detail. Recall that a starting point for Ng\^o's proof of the fundamental lemma was an identification of the compactified Picard scheme for a singular locally planar curve with a classical affine Springer fiber for $GL_n$ (see Definition \ref{def:genzdasf}) \cite{LaumonNgo, MY}. This identification realizes a rank 1 torsion-free module over $R:=\C[[x,t]]/(f)$ as a lattice in the total ring of fractions $\Frac(R)$.

When $f$ is an \emph{irreducible} polynomial, the corresponding classical affine Springer fiber for $SL_n$ can further be related to the compactified Jacobian of the singularity. When the compactified Jacobian is of finite type and admits a $\C^\times$ action, one can construct a Springer-like action of the rational Cherednik algebra (of $\mathfrak{sl}_n$) on its equivariant cohomology using a perverse filtration \cite{OY}. The work of Maulik-Yun and Migliorini-Shende \cite{MY,MS} shows that this perverse filtration arises from the Hilbert scheme of points on the curve via an Abel-Jacobi map.

We take the relation between affine Springer theory of $GL_n$ and Hilbert schemes further by interpreting (flags of) ideals of $R:=\C[[x,t]]/(f)$ as (flags of) lattices in $\Frac(R)$ contained in the standard lattice. These moduli spaces of lattices also have a realization as {\em generalized} affine Springer fibers in the sense of \cite{GKM}. 

Let us recall some basic facts about generalized affine Springer fibers. Choose a reductive group $G$ and let $\cK = \C((t))$. The data of a generalized affine Springer fiber includes a choice of a representation $N \in \text{Rep}(G)$, a parahoric subgroup $\bP \subset G_\cK$, a lattice $\bN \subset N_\cK$ stable under $\bP$, and a vector $v \in N_\cK$. More precisely, they are the reduced fibers of the map 
$$G_\cK\times_{\bP} \bN\to N_\cK$$ given by $(g,n)\mapsto g^{-1}.n$ and can be thought of as affine generalizations of Hessenberg varieties. Note that the generalized affine Springer $M_v$ is naturally identified with a sub(ind-)scheme of the partial affine flag variety $G_\cK / \bP$.\footnote{A more precise notation for the generalized affine Springer fiber $M_v$ would include the choices of the reductive group $G$, the parahoric subgroup $\bP$, and the lattice $\bN$. We suppress this dependence to avoid overburdening the notation.} Our first main result (Theorem \ref{thm:maincomparison} in the main text) is the following.

\begin{theorem}
	\label{introthm:maincomparison}
	Let $\hC:=\Spec R$ be a germ of a plane curve singularity and write $R=\C[[x,t]]/(f)$. If $f$ has $x$-degree $n$ then there is a generalized $\Ad\oplus V$-affine Springer fiber $M_v\subset \Gr_{GL_n}$ so that there is an isomorphism of (ind-)schemes $$\varphi: M_v\to \Hilb^\bullet(\hC),$$ where $v = (\gamma, e_n)$ for $\gamma$ the companion matrix of $f$ and $e_n$ the $n$-th standard basis vector of $V = \C^n$.
\end{theorem}

In the statement of this theorem and the remainder of the paper, $V$ is the vector representation of $GL_n$. We only stated the above theorem in the spherical case to keep the introduction more readable, but prove a more general form in the main body of the text. This generalization shows that parabolic flag Hilbert schemes, see e.g. \cite{GSV}, and the incidence varieties of \cite{ORS}, defined in terms of flag Hilbert schemes of $\hC$, also have natural interpretations as generalized affine Springer fibers.

\begin{remark}
	The proof of Theorem \ref{thm:maincomparison} does not require that the curve $\hC$ is reduced or irreducible. In particular, this construction yields generalized affine Springer fibers realizing the Hilbert scheme of points on non-reduced curves.
\end{remark}

\begin{remark}
While it would be tempting to interpret {\em all} generalized affine Springer fibers for $N=\Ad\oplus V$ as variants of Hilbert schemes of points, a moment's thought shows that this is not possible. For example, if the summand $e$ of $v=(\gamma,e)\in (\Ad\oplus V)_\cK$ is just $e=0$, we recover ordinary affine Springer fibers. In general, the proof of the Theorem shows that the lattice generated by $e, \gamma e, \gamma^2 e,\ldots,\gamma^n e$ should equal the standard lattice $\cO^n$ in order to recover a generalized affine Springer fiber $M_v$ isomorphic to a (parabolic) Hilbert scheme. Note that this is a stronger condition than merely requiring $e$ to be a cyclic vector. Note also the similarity of this construction to the GIT construction of $\Hilb^n(\C^2)$ \cite{Nak}, in which we quotient out the stable locus $\{(X,Y,v)\in \Ad(\C)^2\oplus \C^n|[X,Y]=0, \C[X,Y]v=\C^n\}$ by the natural $GL_n$-action. We do not know whether there is a GIT-style interpretation of Theorem \ref{introthm:maincomparison}.
\end{remark}

\subsection{Generalized affine Springer theory}
The (co)homologies of the classical affine Springer fibers admit an action of the trigonometric double affine Hecke algebra, at least in the ``homogeneous cases" \cite{OY,VV}, similar to the classical Springer action of the graded affine Hecke algebra on the (co)homologies of Springer fibers. Therefore, it is natural to expect that there is a Springer-type action of some algebra on the homologies of the generalized affine Springer fibers as well, for arbitrary $(G,N)$ (see \cite[Remark 3.9.(4)]{BFN}), and in particular in this case of Hilbert schemes of points.

This turns out to be the case, as recently explored by Hilburn-Kamnitzer-Weekes \cite{HKW} in the spherical case. The algebras in question turn out to be (quantizations of) the ring of functions on the {\em Coulomb branch} of a corresponding three-dimensional $\cN=4$ gauge theory, or simply the {\em Coulomb branch algebra}, as mathematically defined in \cite{BFN} by a convolution algebra construction modeled on the affine Grassmannian (and in our case, other partial affine flag varieties), generalizing the work of \cite{BFM} for $N=0$. We generalize the results of \cite{HKW} to their natural maximum, allowing in particular for generalized affine Springer fibers in any partial affine flag variety. Combining our construction with the results of \cite{HKW} will give us more insight into the nature of Springer representations of various algebras arising as Coulomb branches.

It was shown by Kodera-Nakajima \cite{KN} that the Coulomb branch algebra, for the datum $G = GL_n$ and $N = \Ad\oplus V$, where $\Ad$ is the adjoint representation and $V$ is the vector representation, is isomorphic to the spherical rational Cherednik algebra of $\fg\fl_n$. In addition, \cite{WebDO, BEF} prove that the Iwahori version of the Coulomb branch in question is naturally isomorphic to the full rational Cherednik algebra. See Theorem \ref{thm:rca} and subsequent discussion for the precise statements.

Combining the above ingredients, we find an action of the spherical rational Cherednik algebra of $\fg\fl_n$ on the equivariant Borel-Moore homology of $\Hilb^\bullet(\hC)$ as a type of ``generalized affine Springer theory" similar to the orbital variety version in \cite[Section 6.5.]{CG}. The Iwahori generalization of this yields an action of the full rational Cherednik algebra on the parabolic flag Hilbert schemes $\PHilb^{[\bullet,\bullet+(1,\ldots,1)]}(\hC)$. More precisely, subject to a mild constraint on the stabilizer $L_v$ of $v$ (see Theorem \ref{thm:convolution} for more details), we arrive at the following result.

\begin{proposition}
	The rational Cherednik algebra $\cH_n$ of $\fg\fl_n$ acts on $$\bigoplus_{m\geq 0} H_*^{L_v}(\PHilb^{[m,m+(1,\ldots,1)]}(\hC))$$ and the spherical rational Cherednik algebra $e\cH_n e$ acts on $$\bigoplus_{m\geq 0} H_*^{L_v}(\Hilb^m(\hC))$$ via a natural convolution product. 
\end{proposition}

This fits well with the results of \cite{GORS,ORS,EGL,OY,GSV}, see e.g. Section \ref{sec:gsv}, where we compare our calculations with the recent results of Gorsky-Simental-Vazirani.

For the case where the plane curve singularity $\hC = \hC_{n,k}$ is quasi-homogeneous and given by $f=x^n-t^k$, we find the above actions with parameter $m = -\tfrac{k}{n} \hbar$ (to match with most conventions, we must specialize $\hbar \to -1$) on the equivariant Borel-Moore homology with respect to the stabilizer $L_v \cong \C^\times\subset \C^\times_{rot}\times \C^\times_{dil}$ of a specific element $v\in N_\cO$, realizing an expectation of \cite{ORS}. When $\gcd(n,k)=1$, the Hilbert scheme of points $\Hilb^\bullet(\hC_{n,k})$ has isolated $\C^\times$-fixed points and we can take the analysis quite far. We compute the action in the basis of fixed points by means of an ``abelianization procedure" akin to \cite{BDGH,BFNSlices,lineops} in some cases. 
\begin{remark}
	This abelianization rests on the rather general localization formula in Proposition \ref{prop:torusloc}. This general formula may be of independent interest.
	
	For example, the work \cite{HKW} identifies a certain generalized affine (or BFN) Springer fiber with the central fiber of the small resolution $\cL_n \to \mathcal{Z}_n$ of Drinfeld's Zastava space $\mathcal{Z}_n$ by Laumon's space $\cL_n$. Our localization formula could serve as a way of explicitly relating the ``BFN Springer action" on homologies of spaces of quasimaps and the $\mathcal{U}_\hbar \fg\fl_n$-action of Feigin-Finkelberg-Frenkel-Rybnikov on $\cL_n$ \cite{FFFR}.
\end{remark}
In \cite{GORS}, Gorsky-Oblomkov-Rasmussen-Shende (GORS) suggested a connection between the homology of the Hilbert scheme in the quasi-homogeneous case and irreducible representations of the spherical Cherednik algebra. Using the above proposition, we verify an isomorphism between these representations:
\begin{theorem}[Theorem \ref{thm:glnirrep}]
	\label{thm:fdirrep}
	When $\gcd{(n,k)} = 1$, we have $$H^{\C^{\times}}_*(\Hilb^\bullet(\hC_{n,k})) \simeq eL_{k/n}(\textrm{triv})$$ as modules for the spherical rational Cherednik algebra of $\fg\fl_n$.
\end{theorem}

\begin{remark}
	For the case of $(2,2\ell+1)$ torus knots we show this directly, see Section \ref{sec:SCA2}. For the remaining cases the direct analysis becomes cumbersome, so we resort to a dimension computation to conclude the result. It is however remarkable that our approach is, in principle, amenable to completely explicit computation, when compared with e.g. \cite{OY}. We also note that Theorem \ref{thm:fdirrep} is compatible with the earlier results and conjectures of \cite{VV, OY, ORS, GORS} relating modules for the spherical rational Cherednik algebra and $\Hilb^\bullet(\hC_{n,k})$.
\end{remark}
\subsection{HOMFLY-PT homology of torus knots}
The links of the quasi-homogeneous ($\hC_{n,k}=\{x^n=t^k\}$) singularities are the (positive) $(n,k)$-torus links, and it has been known for a while that the representations constructed above are closely connected with corresponding ``lowest $a$-degree parts" of the HOMFLY-PT homologies of these links. In particular, our approach combined with recent results of Hogancamp-Mellit \cite{HM} (and the older philosophies of Gorsky-Oblomkov-Rasmussen-Shende \cite{GORS,ORS}) quite directly shows the fact that the rational Cherednik algebra of $\fg\fl_n$ acts on these link homologies, {\em par transport de structure}. This is the subject of Section \ref{sec:toruslinks}. 
\begin{remark}
	The higher $a$-degrees also have natural interpretations from the parahoric viewpoint, and the full Iwahori invariant, \emph{i.e.} considering full flags of ideals, is likely related to the annular invariant introduced in Trinh's thesis \cite[Definition 1.7.8]{Trinh}. Conditioned on the Oblomkov-Rasmussen-Shende (ORS) conjecture \cite[Conjecture 2]{ORS}, our results also imply the rational Cherednik algebra acts on the HOMFLY-PT homology of {\em any} algebraic link. We do not pursue these directions further.
\end{remark}

\begin{remark}[For the physically minded reader]
	As is clear from the introduction, we were inspired in part by the physics of three-dimensional $\cN=4$ gauge theory \cite{DGHOR} and its relationship to a recent construction of the triply graded HOMFLY-PT homology \cite{OR}, whereby the various $a$-degrees are realized within a certain category of matrix factorizations. 
	
	In the upcoming (companion) work \cite{DGHOR}, the construction of \cite{OR} is interpreted as a computation in the $B$-twist of $U(n)$ gauge theory with hypermultiplets transforming in the representation $T^*N$ for $N = \Ad\oplus V$. For the $\ell$-th possible $a$-degree, one computes the supersymmetric Hilbert space of the theory in the presence of a Wilson line in the representation $\bigwedge^\ell V$ subject to a certain boundary condition whose parameters specify the knot in question.
	
	The three-dimensional mirror of this construction is a computation in the $A$-twist of the same theory. Again, one computes the supersymmetric Hilbert space of the theory but now in the presence of a particular vortex line and subject to a different boundary condition. The parameters of this boundary condition translate to the eigenvalues of one of the adjoint fields, which braid around one another along the boundary. For algebraic links, this computation can be reformulated algebraically and one finds that the supersymmetric Hilbert space associated to the lowest $a$-degree component of HOMFLY-PT homology can be computed as the homology of the generalized affine Springer fibers we discuss below.
	
	In the general context of three-dimensional $\cN=4$ theories, the supersymmetric Hilbert spaces associated to boundary conditions and the action of the quantized Coulomb branch on them appeared previously in \cite{BDGH} and \cite{BDGHK}, and we make their geometric action rigorous via the BFN presentation in Section \ref{sec:sca}. In many cases of interest, we can realize the action of the Coulomb branch using an ``abelianization procedure," c.f. \cite{BFNSlices, BDG, WebKD}.
	
	A generalization of these Hilbert spaces, and the local operators that act upon them, that includes ($\tfrac{1}{2}$-BPS) vortex line operators appeared briefly in \cite{BDGH} and was the central aim of \cite{lineops}. Some choices of vortex lines and boundary conditions admit an algebraic realization; the vortex lines are labeled by a choice of unbroken gauge group $\bP$ together with allowed profiles for the matter fields $\bN$ compatible with $\bP$. For a certain class of Dirichlet boundary conditions, and in particular those used in the construction of HOMFLY-PT homology for algebraic links, the supersymmetric Hilbert space in the presence of a line operator $\cL_{\bP, \bN}$ is identified with the equivariant homology of a generalized affine Springer fiber.
	
	We describe below the line operator $\cL_\ell$, \emph{i.e.} choice of $\bP, \bN$, and Dirichlet boundary condition, \emph{i.e.} choice of fiber, realizing the incidence varieties of \cite{ORS}. ORS conjecture \cite[Conjecture 2]{ORS} that the equivariant homology of these incidence varieties, \emph{i.e.} the supersymmetric Hilbert spaces in the presence of this boundary condition and vortex line, realizes the $\ell$-th possible $a$-degree of HOMFLY-PT homology. These homologies are naturally endowed with an action of the algebra of local operators bound to the vortex line, \emph{i.e.} an action of a convolution algebra generalizing the Coulomb branch construction of BFN \cite{BFN}.
	
	A physical derivation of the choice of line operator $\cL_\ell$ and boundary condition used in our construction, as well as its relation to the work of Oblomkov-Rozansky \cite{OR}, will be discussed in \cite{DGHOR}. Understanding the module structure of these homologies is a direction for future work.
\end{remark}

\begin{remark}
	From the physical perspective described above, it is clear why the action constructed in \cite{OY} is of the \emph{trigonometric} Cherednik algebra, whereas the present work constructs an action of the \emph{rational} Cherednik algebra.
	
	The construction of \cite{OY} describes the action of the Coulomb branch (and a vortex line generalization thereof) for a related gauge theory (a 3d $\cN=4$ gauge theory again with $U(n)$ gauge group but instead with the representation $N = \Ad$ instead of $N = \Ad \oplus V$) and Dirichlet boundary condition. The Coulomb branch of this theory is known to be the spherical subalgebra of the trigonometric Cherednik algebra \cite{KN} and the above physical construction realizes an action thereof on the homology of an ``classical" affine Springer fiber.
	
	Mathematically, the geometry involved in the BFN construction shows that the Coulomb branch algebra for a representation $N=N_1\oplus N_2$ injects into the Coulomb branch algebra of $N_i$. From \cite{KN,BEF} these injections for the Jordan quiver with different framings are identified with Suzuki's embedding of the type A cyclotomic Cherednik algebra into the trigonometric one.

From our construction, one also has a similar inclusion of generalized affine Springer fibers for ``forgetting" representations. Although we do not use it in the present work, it seems plausible that these inclusions induce equivariant maps for the convolution actions of the Coulomb branch algebras in Borel-Moore homology. The case of $\mathbf{P} = G(\mathcal{O})$ and $\bN = N(\mathcal{O})$ is provided \cite[Prop. 4.15]{HKW}.
\end{remark}

\begin{remark}
	Most of our results, including the computations with fixed-point localization, make sense over other algebraically closed fields, in particular $\overline{\F}_q$ with $\overline{\Q}_\ell$-coefficients in cohomology. But since it makes life easier, and the results of \cite{BFN} are also written in the language of algebraic geometry over $\C$, we have decided to work over $\C$ throughout. This also makes the comparison to link homology more transparent.
\end{remark}
The paper is organized as follows. In Section \ref{sec:gasf} we recall the necessary definitions of generalized affine Springer fibers $M_v$. In Section \ref{sec:hilb} we identify the generalized affine Springer fiber (for the datum $(GL_n,\Ad\oplus V)$) isomorphic to $\Hilb^\bullet(\hC)$, and generalizations thereof, for $\hC$ the germ of a plane curve singularity. In Section \ref{sec:sca} we define a convolution action of the quantized Coulomb branches of \cite{BFN} on the equivariant (Borel-Moore) homology of the generalized affine Springer fibers $M_v$, specializing in particular to the action of the spherical rational Cherednik algebra on the equivariant homology of the Hilbert schemes. The proof that the convolution really defines an action is relegated to Appendix \ref{appendix}. In Section \ref{sec:toruslinks} we discuss the quasi-homogeneous singularities $\hC_{n,k}$ related to $(n,k)$ torus links and show how they relate to rational Cherednik algebra representations. In Section \ref{sec:SCA2} we discuss $(2,2\ell+1)$ torus knots in detail.

\section{Generalized Affine Springer Theory}
\label{sec:gasf}
This section is written in more generality than is needed for most of our main results.
Let $G/\C$ be a reductive group, $\fg=\Lie(G)$, and $N$ be an algebraic representation of $G$. Let $\cK=\C((t))$ and $\cO=\C[[t]]$. 
Let $\bP$ be a parahoric subgroup of $G(\cK)$ and $\bN\subset N(\cK) $ a lattice stable under $\bP$. In later sections, we only use standard parahorics $\bP\subset G(\cO)$ coming as preimages of parabolic subgroups in $G(\C)$ via the ``evaluation at zero" map, but it should be clear where this assumption can be dropped.
Let $\Gr_G$ be the affine Grassmannian of $G$, $\Fl_G$ the affine flag variety of $G$ and, 
more generally, $\Fl_\bP$ the partial affine flag variety associated to $\bP$. On the level of $\C$-points, $\Fl_\bP(\C)=G(\cK)/\bP$.

\begin{definition}
	\label{def:genzdasf}
	Let $v\in N(\cK)$. Define the {\em generalized affine Springer fiber} (GASF) associated to the datum $(v,\bP,\bN)$ as the closed sub-ind-scheme of the partial affine flag variety $\Fl_\bP$ whose functor of points is defined as
	$$M_v^{\bP, \bN}(A):=\{g\in G(A((t))) | g^{-1}.v\in \bN(A)\}/\bP(A)$$ for any $\C$-algebra $A$.
\end{definition}

\begin{remark}
	Note that the definition of $M^{\bP, \bN}_v$ also depends on $G$. Since we will only be working with $G=GL_n$, we mostly omit these from the notation. When $\bP=G(\cO), \bN = \Ad(\cO) \oplus \cO^n$, we simply denote $M_v^{\bP, \bN}$ by $M_v$. Similarly, when $\bP=\bI,\bN = \Lie(\bI) \oplus \cO^n$ for $\bI$ an Iwahori subgroup we use $\tM_v$ and, more generally, when $\bN = \Lie(\bP) \oplus \cO^n$ we use $M^\bP_v$. 
\end{remark}
\begin{remark}
Note that $M_v$ is in general highly nonreduced. We will only work with its etale/singular Borel-Moore homologies so may in practice work with the reduced structure only.

\end{remark}
\begin{remark}
	The ``classical" affine Springer fibers are the case when $N=\Ad$ and $\bN$ is the Lie algebra of $\bP$. As explained in \cite{GKM}, the GASF can be thought of as an affine analog of Hessenberg varieties. Note that both our GASF and those of \cite{GKM} are different from the Kottwitz-Viehmann varieties, which are group versions of affine Springer fibers.
\end{remark}

In \cite{Va,VV,OY}, an action of the (degenerate) double affine Hecke algebra of $\fs\fl_n$ was constructed on the equivariant (K-)homology of certain (usual) affine Springer fibers using the convolution algebra technique (see e.g. \cite{CG}).

Just as affine Springer fibers are a source of affine Springer representations of affine Weyl groups and Cherednik algebras, generalized affine Springer fibers can be used to construct representations of certain convolution algebras associated to the datum $(G,N)$, see \cite[Theorem 3.10]{BFN} or Theorem \ref{thm:BFN} below. These are the ``quantized Coulomb branches" of three-dimensional $\cN=4$ field theories, or ``BFN algebras." In the classical case $N=\Ad$, the K-theoretic analog of the Coulomb branch algebra is the DAHA, as explained e.g. in \cite{FT}.

In particular, in \cite{HKW}, the convolution algebra technique from above was extended to {\em any} Coulomb branch algebra. The authors of {\em loc. cit.} were kind enough to share their preliminary results on the topic with us, and we expand upon these results in Section \ref{sec:sca} (which focuses on the $N=\Ad\oplus V$ case) and in Appendix \ref{appendix}. We also define the maximal parahoric generalization of the generalized affine Springer theory, using natural variations of the techniques in \cite{BFN,HKW}.

\begin{remark}
	In analogy with \cite{YunThesis}, we expect there to be a ``global" Springer theory defined on certain generalized Hitchin spaces (spaces of quasimaps) at least for $N$ with good invariant-theoretic properties. This direction will be pursued in future work.
\end{remark}

\section{Hilbert Schemes of Points on Curve Singularities}
\label{sec:hilb}
Let $\hC:=\Spec R$ be the germ of a (possibly non-reduced) plane curve singularity and write $R=\C[[x,t]]/(f)$ for $f \in \C[[x,t]]$ with $f(0,0) = 0$.
\begin{remark}
	\label{rmk:weierstrass}
	We note that many choices of $f$ yield the same $\hC$. In particular, we may always use the Weierstrass preparation theorem to choose an $f$ that is Weierstrass, \emph{i.e.} a monic polynomial in $x$ whose coefficients are formal series in $t$ that vanish at $t = 0$:
		$$f(x,t) = x^n - a_{n-1} x^{n-1} - ... - a_0 \qquad a_k(t) \in t \C[[t]]\,.$$
	 We will usually denote the $x$-degree of $f$ by $n$.
\end{remark}
\begin{definition}
	The {\em Hilbert scheme of $m$ points on $\hC$} is defined as the scheme representing the functor of points
	$$\hC^{[m]}(A):=\Hilb^m(\hC)(A):=\{\text{colength } m \text{ ideals in } A[[x,t]]/f\},$$ where $A$ is any $\C-$algebra.
	Similarly, given a partition $\vec{p} = (p_1, ..., p_d)$ of $n$, the  {\em $\vec{p}$-flag Hilbert scheme of $m+n$ points on $\hC$} is defined as the scheme representing $$\begin{aligned}
		\hC^{[m,m+\vec{p}]}(A) & :=\Hilb^{[m,m+\vec{p}]}(\hC)(A)\\
		& :=\{I_d \subset ... \subset I_0 \subset A[[x,t]]/f  | I_i \text{ is a colength } m + \sum\limits_{j=1}^i p_j \ \text{ ideal in } A[[x,t]]/f\}\,.
	\end{aligned}$$
\end{definition}
	In particular, the scheme 
	$$\Hilb^\bullet(\hC):=\bigsqcup_{m\geq 0}\Hilb^m(\hC)$$ 
	is naturally the moduli space of finite length subschemes on $\hC$, whereas
	$$\Hilb^{[\bullet,\bullet+\vec{p}]}(\hC):=\bigsqcup_{m\geq 0}\Hilb^{[m,m+\vec{p}]}(\hC)$$ 
	is naturally the moduli space of flags of such subschemes.

	Requiring flags of ideals such that $I_d = t I_0$ puts a natural constraint on the allowed partitions $p$; if $f$ is a polynomial in $x$ of degree $n$ then $p$ must be a partition of $n$. When $\vec{p}=(1,\ldots,1)$ is the one-column partition of $n$, the relevant Hilbert scheme is the {\em parabolic flag Hilbert scheme} $\PHilb^\bullet(\hC)$ (see e.g. \cite{GSV}), consisting of full flags of ideals of length $n$, with the condition that $I_n=tI_0$. More generally, if $\vec{p}=(p_1,...,p_d)$ is any partition of $n$ we can define the {\em $\vec{p}$-parabolic flag Hilbert scheme} $\PHilb^{[\bullet,\bullet+\vec{p}]}(\hC)$.

\begin{definition}
	\label{def:parabolic}
	The {\em $\vec{p}$-parabolic flag Hilbert scheme} $\PHilb^{[m,m+\vec{p}]}(\hC)$ is defined as the scheme $$\PHilb^{[m,m+\vec{p}]}(\hC):=\{I^\bullet \in \Hilb^{[m,m+\vec{p}]}(\hC) | I_d = t I_0)\}.$$
\end{definition}
We now state and prove our first main theorem. As above, we fix the pair $G = GL_n, N = \Ad \oplus V$ and the (germ of a) planar curve singularity $\hC$. Write $\hC = \Spec R$ for $R = \C[[x,t]]/(f)$ and $f$ as in Remark \ref{rmk:weierstrass} and denote $v = (\gamma, e_n) \in N(\cK)$, where $\gamma$ is the companion matrix of $f$ and $e_n$ is the $n$-th standard basis vector of $V = \C^n$.
\begin{theorem}
	\label{thm:maincomparison}
	For the germ of any plane curve singularity $\hC$, the generalized $\Ad\oplus V$-affine Springer fiber $M_v\subset \Gr_{G}$ admits an isomorphism of schemes $$\varphi: M_v\to \Hilb^\bullet(\hC).$$
	More generally, there exist $\bP, \bN = \Lie(\bP) \oplus \cO^n$ such that the generalized $\Ad\oplus V$-affine Springer fiber $M^{\bP}_v\subset \Fl_\bP$ admits an isomorphism of schemes $$\varphi_\bP: M^{\bP}_v \to \PHilb^{[\bullet,\bullet+\vec{p}]}(\hC).$$
\end{theorem}
\begin{proof}
	
	Note that we can interpret $\hC$ and $\hC^{[m]}$ as follows. We use Weierstrass preparation to write $f(x,t)$ as a degree $n$ polynomial in $x$ as in Remark \ref{rmk:weierstrass}. There is an isomorphism of $\C[[t]]=\cO$-modules
	\begin{equation}
	\label{eq:moduleiso}
	R = \C[[x,t]]/(f) \cong \langle 1,x,\ldots, x^{n-1}\rangle_\cO,
	\end{equation}
	where $\langle S \rangle_\cO$ denotes the free $\cO$-module generated by a set $S$. (For an arbitrary $\C$-algebra $A$ we also have $A[[x,y]]/(f)\cong\langle 1,x,\ldots, x^{n-1}\rangle_{A[[t]]}$, and similar considerations apply below, so we omit this from the notation.)
	
	Taking the total ring of fractions of $R$, we see that as $\C((t))=\cK$-vector spaces
	$\Frac(R)\cong (\cK^n)^*$ ($\cK$-linear dual of $\cK^n$) as follows. If $f$ is square-free so that $\hC$ is reduced, $\Frac(R)\cong \prod_{i=1}^dF_i$  where $d$ is the number of irreducible factors over $\cK$ of $f$ and $F_i$ are finite extensions of $\cK$ so that $\sum_i [F_i:\cK]=n$.
	
	If $f$ has a repeated factor, by the Chinese Remainder Theorem we have an isomorphism $R\cong \prod_{i=1}^d\cO_i$ where each $\cO_i$ is a finite ring extension of $\cO$ which is torsion-free over $\cO$. More precisely, writing $f=
\prod_{i=1}^d f_i^{m_i}$, where $f_i$ are irreducible and pairwise distinct, $\cO_i\cong \cO[x]/f_i^{m_i}$.  Since $\cO$ is a domain, $\Frac(\cO_i)\cong \cO_i\otimes_\cO \cK$. As a $\cO$-module, $\cO_i\cong \cO^n$ and in particular, $\Frac(R)\cong (\cK^n)^*$.
	
	Under the identification $\phi^*:\Frac(R)\cong (\cK^n)^*$, the natural injection $R\hookrightarrow \Frac(R)$ realizes the isomorphism in \eqref{eq:moduleiso} by identifying $R$ with $(\cO^n)^*$ and $1\in R$ with the vector $e_1^*=(1,0,\ldots, 0)$ in $(\cK^n)^*$. We may moreover choose $\phi^*$ so that in the dual basis of $(\cK^n)^*$, the operator of multiplication by $x$ has the form $$\gamma=\begin{pmatrix}
	0 & 1 & \cdots & 0 & 0\\
	\vdots & \vdots & \ddots & \ddots & \vdots\\
	0 & 0 & \ddots & 1 & 0\\
	0 & 0 & \cdots & 0 & 1\\
	a_0 & a_1 & \cdots & a_{n-2} & a_{n-1}\\
	\end{pmatrix},$$
	where $a_i \in t\cO$. Recall that a matrix of the above form is called the {\em companion matrix} of the polynomial $x^n-a_{n-1} x^{n-1} -\cdots-a_1 x-a_0$. In particular, $f(x,t)$ agrees with the characteristic polynomial of $\gamma$ $$f(x,t) = \det(x \text{Id}_{n} - \gamma).$$ Note that $\{e_k^* = e_1^*\gamma^{k-1}\}_{k=1}^{n}$ is a $\cO$-basis of $(\cO^n)^*$ and $\gamma e_{k+1} = e_k + a_k e_n$ for $k = 1,..., n-1$.
	
	By definition, $\cO$-lattices in $(\cK^n)^*$ stable under $\gamma$ are the same as (nonzero) fractional $R$-ideals, i.e. $R$-submodules $\Lambda$ with nonzero $r\in \Frac(R)$ with $r\Lambda\subset R$.
	The variety of nonzero ideals of finite codimension in $R$ is then identified with fractional ideals in $\Frac(R)$ contained in $R$. Indeed, note that the condition of being a lattice implies that tensoring $\Lambda$ with $\cK$ and projecting to each factor of $\cK$ is a surjective map, hence the corresponding ideal is of finite codimension. Under $\phi$, we get $$\Hilb^\bullet(\hC)\cong X:=\{\Lambda\subset (\cO^n)^*| \Lambda\gamma \subset \Lambda\}.$$ 
	
	Now for any lattice $\Lambda$, there is an element $g\in G(\cK)$ so that 
	$\Lambda=(\cO^n)^* g^{-1}$.
	It is well defined up to the stabilizer of $(\cO^n)^*$, which is $G(\cO)$. 
	If $\Lambda \subset (\cO^n)^*$ and $\Lambda \gamma \subset \Lambda$, we have 
	\begin{enumerate}
		\item $g^{-1}\in G(\cK) \cap \fg\fl_n(\cO)$, because $(\cO^n)^* g^{-1} =\Lambda \subset (\cO^n)^*$, and
		\item $g^{-1} \gamma g \in \Ad(\cO)$, because $(\cO^n)^* g^{-1} \gamma g=\Lambda \gamma g \subset \Lambda g = (\cO^n)^*$ and the stabilizer of $(\cO^n)^*$ is $\fg \fl_n(\cO) = \Ad(\cO)$.
	\end{enumerate}
	If $e_i$ denotes the standard basis in $\cK^n$, the first point implies that $g^{-1} e_n$ belongs to $\cO^n$.
	
	Let $v:=(\gamma, e_n) \in (\Ad \oplus V)(\cO)$ and consider the map $$\Lambda\mapsto [g]$$ from $X$ to the scheme 
	$$M_v=\{[g]\in \Gr_G | g^{-1} \gamma g\in \Ad(\cO), g^{-1} e_n \in \cO^n\}.$$ 
	We will construct an inverse to this map. Given any $[g]\in M_v$, we have
	\begin{enumerate}
		\item $g^{-1}\in G(\cK) \cap \fg\fl_n(\cO)$, because $g^{-1} e_n\in \cO^n$, $g^{-1} \gamma g \in \Ad(\cO)$ and $$g^{-1} e_{k}= (g^{-1} \gamma g) g^{-1} e_{k+1} - a_k g^{-1} e_n \in \cO^n$$ for $k = 1, ..., n-1$, and 
		\item $(\cO^n)^* g^{-1} \gamma \subset (\cO^n)^* g^{-1}$, because $g^{-1} \gamma g\in \Ad(\cO)$. 
	\end{enumerate}
	The first point implies that $\Lambda = (\cO^n)^* g^{-1} \subset (\cO^n)^*$ and the second implies $\Lambda$ is closed under the action of $\gamma$, i.e. $\Lambda \in X$. As these constructions are inverse to each other, we have $X\cong M_v$.
	
	Finally, composing with the isomorphism to $\Hilb^\bullet(\hC)$ we get that 
	$$\Hilb^\bullet(\hC)\cong M_v.$$ By Definition \ref{def:genzdasf} the space $M_v$ is the generalized $\Ad\oplus V$-affine Springer fiber for $v=(\gamma, e_n)$.
	
	Now choose a partition $\vec{p} = (p_1, ..., p_d)$ of $n$ and let $\bP$ be the corresponding parahoric subgroup. From the above we know that a flag of ideals $t I_0 = I_d \subset ... \subset I_0 \subset R$ can be identified with a flag of lattices $t \Lambda_0 = \Lambda_d  \subset ... \subset \Lambda_0 \subset (\cK^n)^*,$ such that each lattice is closed under the action of $\gamma$. Such a flag is the $G_\cK$ translate of the standard flag of lattices $t (\cO^n)^* \subset ... \subset (\cO^n)^* \subset (\cK^n)^*$, and the stabilizer of this standard flag is exactly $\bP$, \emph{i.e.} each of these flags of lattices can be identified with some $[g] \in \Fl_\bP$. For $[g] \in \Fl_\bP$ to realize a flag of ideals is equivalent to $g^{-1}.v \in \Lie(\bP)\oplus \cO^n$. Just as above, the identification $t I_0 \subset ... \subset I_0 \subset R \leftrightarrow [g]$ yields the desired isomorphism with $M^{\bP}_v$.
\end{proof}
\begin{remark}
	\label{rmk:incidence}
	It is interesting to consider the generalized affine Springer fiber over the same $v$ as above but with $\bN \neq \Lie(\bP)\oplus \cO^n.$  One such variant yields the incidence varieties ``$C^{[m \leq m + l]}$" (note the notational difference to this paper) of \cite{ORS}, where we choose the partition $(l, n-l)$ and require that the $\Ad(\cO)$ element is proportional to $t$ in the first $l$ columns:
	$$\bN = \mathfrak{gl}_n(\cO) t^{(1,...,1,0,...,0)} \oplus \cO^n\,.$$ 
	This choice of $\bN$ ensures that the flag of lattices $t \Lambda_0 \subset \Lambda_1 \subset \Lambda_0$ satisfies $\Lambda_0 \gamma \subset \Lambda_1$. In terms of ideals, this latter point implies that $M I_0 \subset I_1 \subset I_0$, where $M = \langle x, t \rangle$ is the maximal ideal of $R$. See \cite{DGHOR} for more details. Note that in the $G=SL_n, N=\Ad$-case similar incidental varieties appear in the work of Cherednik and Philipp \cite{CP} under the name of flagged Jacobian factors.
\end{remark}
\begin{remark}
	An equivalent, perhaps preferred, description of $\Hilb^\bullet(\hC)$ is as lattices $\Lambda \subset \cO^n$. If we identify $1 \leftrightarrow e_1$, then following the above proof one finds an isomorphism to the generalized $\Ad\oplus V^*$-affine Springer fiber $M_w'$ for the vector $w=(\gamma^T, e_n^*) \in \Ad(\cO)\oplus (\cO^n)^*$, c.f. \cite{YunPCMI}.
\end{remark}
\begin{remark}
Note that the proof doesn't assume $\hC$ to be reduced. In particular, this suggests us to define the ``compactified Picard variety" $\bPic(\hC)$ for these non-reduced curves as the classical $\GL_n$-affine Springer fiber, although it is usually not considered in the literature. For example, when $\gamma$ is the regular nilpotent matrix, the ASF in question gives an infinite-dimensional affine Springer fiber whose homology coincides with that of the affine Grassmannian. Similarly, the GASF in question yields the Hilbert schemes of points on the non-reduced curve $\{x^n=0\}$, which are now finite-dimensional projective subvarieties of the ``positive part" of the affine Grassmannian.
\end{remark}
\begin{remark}
\label{rmk:positivegrassmannian}
More generally, note that by $$\Hilb^\bullet(\hC)\cong \{\Lambda\subset (\cO^n)^*| \Lambda\gamma \subset \Lambda\}$$ we may identify $\Hilb^\bullet(\hC)$ as the intersection $$\Sp_\gamma\cap \Gr_{GL_n}^+$$ where $\Sp_\gamma$ is the ``usual" ($N=\Ad$) affine Springer fiber of $\gamma$ and $\Gr_{GL_n}^+$ is the positive part of the affine Grassmannian $$\Gr_{GL_n}^+:=\{\Lambda \subseteq (\cO^n)^* \subset (\cK^n)^*\}$$ not to be confused with the ``positive Grassmannian" which is a distantly related object of intense research. See also \cite[Remark 4.24]{Kiv2}.
\end{remark}
\begin{remark}
	Using the decomposition of $\Gr_G$ by $\pi_1(G)=\Z$ we find that $M_v$ can be expressed as
	$$M_v = \bigsqcup_{m \leq 0}M_v^m,$$
	where $M_v^m$ is the component of $M_v$ inside the degree $m$ part of $\Gr_G$. Indeed, we have $M_v^{m} = \Hilb^{|m|}(\hC)$. Thus $M_v$ is a (infinite) disjoint union of projective varieties, because the Hilbert scheme for fixed $m$ can be realized as a closed subvariety of a Grassmannian. There is a similar decomposition of $M_v^{\bP, \bN}$ obtained from the decomposition of $\Fl_\bP$ by $\pi_1(G)$, coming via pullback by the projection $\Fl_\bP\to \Gr_G$.
\end{remark}

\subsection{Links and torus actions}
If $f(x,t)$ is a polynomial, we may interpret $\hC$ as the germ of the curve $C=\{f=0\}\subset \C^2$. In this case, the intersection of $C$ with a small three-sphere centered at the origin yields a compact one-manifold 
$$\cL:=\text{Link}_0(C)\hookrightarrow S^3.$$

By work of Oblomkov-Rasmussen-Shende and others (see \cite{Mig} and references therein) it is conjectured that, topologically, the Hilbert schemes of $\hC$ are controlled by the HOMFLY-PT homology of the corresponding link $\cL$. See Conjecture \ref{conj:ORS} for a more precise statement of the conjecture for minimal $a$-degree.

Consider $f$ of the form $f = x^n - t^k$ for $n,k \geq 0$. The special form of $f$ in this case means that the singularity is {\em quasi-homogeneous}, so there is a straightforward $\C^\times$ action on $\cM_{(n,k)}:=M_v$ coming from scaling $x$ and $t$. As has been noted by various authors, we thus get an extra torus action on the Hilbert schemes. This is less trivial on the generalized affine Springer fiber side. 

Namely, let $1\to G\to \tG\to G_F\to 1$ be an extension of algebraic groups over $\C$ and let $\tG^{\cO}_{\cK}$ be the preimage in $\tG_\cK$ of $G_{F,\cO}$. With our definition of $M_v$, we always have an action of the stabilizer of $v$ in $\tG^\cO_\cK \rtimes \C^\times_{rot}$ on $M_v$ (see the next section). Let $G=GL_n, G_F=\C^\times_{dil}, \tG=GL_n\times \C^\times_{dil}$, where $\C^\times_{dil}$ acts by dilating the $\Ad$-part in $\Ad\oplus V$. This action is considered in \cite{OY} in the case of usual affine Springer fibers, where $\C^\times_{rot},\C^\times_{dil}$ are denoted $\G_m^{rot}, \G_m^{dil}$. For $v=(\gamma,e_n)$ corresponding to $f=x^n-t^k$ as in Theorem \ref{thm:maincomparison}, the stabilizer is given as follows. It is worth noting that we use different conventions from the usual (physical) conventions used for $\C^\times_{rot}$ in some of the literature \cite{BDG, BDGH, BDGHK, BFN} . In particular, we do \emph{not} include the overall scaling of $N$ by weight $\tfrac{1}{2}$ in addition to scaling $t$. These conventions are those used by Webster, see e.g. \cite{WebKD}.

\begin{lemma}
	\label{lemma:nkstab}
	For $v=(\gamma,e_n)$ corresponding to $f=x^n-t^k$, with $\gcd(n,k) = 1$, as in Theorem \ref{thm:maincomparison}, we have
	$$L_v:=\text{Stab}_{\tG_\cK^\cO\rtimes \C^\times_{rot}}(v)\cong \C^\times.$$
\end{lemma}
\begin{proof}
	Consider acting with $(g, \mu,\lambda) \in \tG_{\cK}^\cO\rtimes\C^\times_{rot}$ on $v=(\gamma, e_n)$ for $v$ corresponding to $f=x^n-t^k$. Here $\mu$ denotes the flavor part of $\tilde{g}=(g,\mu)\in\tG_\cK^\cO$. Preserving the determinant of $\gamma$ imposes the equation 
	$$\mu^n \lambda^k = 1.$$
	Preserving $e_n$ then says that the last column of $g$ is $e_n$, thus the last column of $g^{-1}$ is also $e_n$. From this, we find that the last column of $g \gamma g^{-1}$ is the penultimate column of $g \mu$, so we need this column of $g$ to be $\mu^{-1} e_{n-1}$ for $g$ to preserve the last column of $\gamma$. This process continues column-by-column so we must have 
	$$g = \diag(\mu^{1-n}, \ldots, \mu^{-1}, 1).$$
	In particular, the stabilizer is the image of the cocharacter $\C^\times\to \tG_{\cK}^\cO\rtimes\C^\times_{rot}$ given by $$\nu\mapsto (\diag(\nu ^{(n-1)k}, \ldots, \nu^k,1), \nu^{-k}, \nu^n).$$
\end{proof}

\begin{remark}
	In general, \emph{i.e.} when $\gamma$ is not quasi-homogeneous, it's {\em always} the case that the stabilizer is trivial by a similar argument. On the other hand, the same proof shows that $\gamma$ for the curve $\{x^n=0\}$ has stabilizer $(\C^\times)^2$ given by $(\diag(\mu^{1-n}, \ldots, \mu^{-1}, 1), \mu, \lambda)$.
\end{remark}

%

\begin{proposition}
	\label{prop:fixedpoints}
	In the case $\gcd(n,k)=1$, the action of $L_v$ on $M_v$ has isolated fixed points labeled by cocharacters $A$ of the maximal torus $T\subset GL_n$ such that
	\begin{equation}
	\label{eq:fixedpointsconstaint}
	\langle A, \omega_n\rangle \geq 0 \hspace{1cm} \langle A, \alpha_i \rangle \geq 0 \hspace{1cm} \sum_{i=1}^{n-1} \langle A, \alpha_i \rangle \leq k,
	\end{equation}
	where $\omega_n$ is the $n$-th fundamental weight of $GL_n$, $\alpha_i$ are the simple roots of $GL_n$, and $\langle, \rangle$ is the pairing of cocharacters and weights. Note that this is just the $k$-dilated fundamental alcove as appears e.g. in \cite{GSV}.
\end{proposition}
\begin{remark}
	If we write $A = (A_1, ..., A_n)$ the above constraint \eqref{eq:fixedpointsconstaint} corresponds to
	$$0 \leq A_n \leq A_{n-1} \leq \ldots \leq A_1 \leq A_n + k$$ This fixed point corresponds to the ideal generated by $(t^{A_1}, t^{A_{2}}x, ..., t^{A_{n-1}} x^{n-2}, t^{A_n} x^{n-1}).$ In this language, the constraint on $A$ is to ensure that this is indeed an ideal. Namely, the set generated by the monomials $t^{A_i} x^{i-1}$ over $\cO$ is closed under multiplication by $x$.
\end{remark}
\begin{proof}
	The action of $\nu \in L_v$ on $[g] \in M_v$ is simply $[\nu g]$, where the product of $L$ and $G(\cK)$ is viewed within $\tG(\cK)\rtimes \C^\times$. In particular, we have
	$$\nu.[g(t)] = [\nu^{((n-1)k,..., k,0)} g(\nu^n t)]$$ 
	where $\nu^{(m_1, ..., m_n)} := \diag(\nu^{m_1}, ..., \nu^{m_n})$. 
	Define the ``orbital variety" (see the next section for motivation) 
	$$\hat{\cV}^v:=G_\cK.v\cap N_\cO.$$
	
	We now describe $M_v \cong \hat{\cV}^v/G(\cO)$. By the Iwasawa decomposition of $G(\cK)$, we can choose to represent elements of $Gr_G$ by a lower-triangular matrix in $G(\cK)$ of the form $h = t^{-A} + q$, where $q$ is strictly lower triangular. Moreover, we can always use $G(\cO)$ to make the (non-zero) $q_{ij}$ Laurent polynomials and with no terms of degree larger than $-A_i-1$, c.f. \cite{LS}. We interpret $A$ as a cocharacter of $T\subset GL_n$.
	
	Note that the chosen representative $h(t) = t^{-A} + q$ is not invariant under $L_v$ but will require a compensating $G(\cO)$ transformation. In particular, under the action of $\nu$, the diagonal entries of $h$ transform as $t^{-A_i} \mapsto \nu^{(n-i)k - n A_i} t^{-A_i}$ whereas $q_{ij}(t) \mapsto \nu^{(n-i)k} q_{i j}(\nu^n t)$. We can always return the diagonal entries to $t^{-A_i}$ by means of a compensating diagonal $G(\cO)$ transformation, sending $\nu^{(n-i)k} q_{i j}(\nu^n t) \mapsto \nu^{k(j-i)+n A_j} q_{i j}(\nu^n t)$. Since the non-zero entries of $q$ are (Laurent) polynomials and have degree at most $-A_i-1$ in row $i$, it follows that there is no lower-triangular matrix that can send this back to $h$. For example, when $j = i-1$ we must solve the equation 
	$$\nu^{n A_{i-1}-k} q_{i i-1}(\nu^n t) + t^{-A_i} p_i(t) = q_{i i-1}(t)$$
	for $p_i(t) \in \cO$. This requires $t^{A_i}(q_{i i-1}(t)-\nu^{n A_{i-1}-k} q_{i i-1}(\nu^n t))$ to belong to $\cO$, hence $$q_{i i-1}(t)-\nu^{n A_{i-1}-k} q_{i i-1}(\nu^n t) = 0$$ since $q_{i i-1}$ has no terms of degree more than $-A_{i}-1$. Finally, since $k$ is coprime to $n$ we conclude that $q_{i i-1}(t) = 0$. With $q_{i i-1} =0$, it is straightforward to inductively show that $q = 0$.
	
	Finally, we see that $t^A e_n \in \cO^n$ if and only if $A_n \geq 0$. Similarly, we have $$t^{A}\gamma t^{-A} = \begin{pmatrix}
	0 & t^{A_1 - A_2} & \cdots & 0 & 0\\
	\vdots & \vdots & \ddots & \ddots & \vdots\\
	0 & 0 & \ddots & t^{A_{n-2} - A_{n-1}} & 0\\
	0 & 0 & \cdots & 0 & t^{A_{n-1} - A_n}\\
	t^{k+A_n - A_1} & 0 & \cdots & 0 & 0\\
	\end{pmatrix},$$
	which belongs to $\Ad(\cO)$ if and only if $A_{i} \geq A_{i+1}$ for $i = 1, ... n-1$ and $A_n + k \geq A_1$. Thus, $t^{A}.v$ to belong to $\hat{\cV}^v$, for $v$ corresponding to the $(n,k)$ torus knot, if and only if
	$$\langle A, \omega_n\rangle \geq 0 \hspace{1cm} \langle A, \alpha_i \rangle \geq 0 \hspace{1cm} \sum_{i=1}^{n-1} \langle A, \alpha_i \rangle \leq k.$$
\end{proof}
\begin{remark}
	When $n$ and $k$ are \emph{not} coprime it is possible to have $$q_{i i-1}(t)-\nu^{n A_{i-1}-k} q_{i i-1}(\nu^n t) = 0$$ for $q_{i i-1}(t)$ nonzero. In these circumstances there are still fixed points but they need not be isolated. 
\end{remark}
\begin{remark}
	The above proof works, up to Weyl group elements, for $L_v$ acting on $\tM_v = \{[g] \in \Fl_\cI | g^{-1}v \in \Lie(\bI) \oplus \cO^n\}$ with $\bI$ an Iwahori subgroup. In particular, when $\gcd{(n,k)} = 1$ there are isolated fixed points which can be represented by matrices $h = t^{-A} \sigma^{-1}$ for cocharacters $A$ of the maximal torus of $T \subset GL_n$ and Weyl group elements $\sigma \in  \fS_n$. For $h^{-1}.v = \sigma t^{A}.v$ to belong to $\Lie(\bI)\oplus \cO^n$, the non-negative integers $(A_1, ..., A_n)$ must have forced jumps. In particular, $\sigma t^A e_n \in \cO^n$ imposes $A_n \geq 0$ and $\sigma t^A \gamma t^{-A} \sigma^{-1} \in \Lie(\bI)$ requires
	$$A_{i} \geq \begin{cases}
	A_{i+1}+1 & \sigma(i+1) < \sigma(i)\\
	A_{i+1} & \sigma(i+1) > \sigma(i)\\
	\end{cases}$$
	$i \in \{1,2,..., n-1\}$ and
	$$A_{n}+k \geq \begin{cases}
	A_1+1 & \sigma(1) < \sigma(n)\\
	A_1 & \sigma(1) > \sigma(n)\\
	\end{cases}$$
	for $i = n$.
	In comparison to the discussion in \cite{GSV}, the fundamental class of this fixed point in equivariant Borel-Moore homology of $\tM_v$ (after localization) corresponds to their ``renormalized" vector $\tilde{v}_{\sigma(A)}$.
\end{remark}
\begin{proposition}
	\label{prop:iwahorifixedpoints}
	In the case $\gcd(n,k)=1$, the action of $L_v$ on $\tM_v$ has isolated fixed points labeled by cocharacters $A$ of the maximal torus $T\subset GL_n$ and $\sigma \in \fS_n$ such that
	\begin{equation}
	\langle A, \omega_n\rangle \geq 0 \hspace{1cm} \langle A, \alpha_i \rangle \geq \tau(i) \hspace{1cm} \sum_{i=1}^{n-1} \langle A, \alpha_i \rangle \leq k - \tau(n).
	\end{equation}
	where $\tau(i) = 1$ if $\sigma(i+1) < \sigma(i)$ and $\tau(i) = 0$ if $\sigma(i+1) > \sigma(i)$, with $\sigma(n+1):=\sigma(1)$.
\end{proposition}

\section{Action of the Rational Cherednik Algebra}
\label{sec:sca}
In this section, we construct an action of the rational Cherednik algebras on equivariant Borel-Moore homologies of Hilbert schemes of $\hC$ and some of its variants.

We first recall the construction of the BFN algebras in general. This is a minor parahoric variant of the construction in \cite{BFN}; the case when $\bP$ is an Iwahori subgroup appears in the work of Webster \cite{WebKD}. We choose a parahoric subgroup $\bP \subset G_\cK$. Suppose $1\to G\to \tG\to G_F\to 1$ is an extension of algebraic groups and suppose 
$\tbP$ is a parahoric subgroup of $\tG_\cK$ that fits in an extension $1\to \bP\to \tbP\to (G_F)_\cO\to 1$ such that $\tbP\cap G_\cK=\bP$. Let
$\tG^{\cO}_{\cK}$ be the preimage in $\tG_\cK$ of $G_{F,\cO}$. 

Note that $$\Fl_\bP\cong \tG^{\cO}_{\cK}/\tbP\cong 
(\tG^{\cO}_{\cK}\rtimes \C^\times_{rot})/(\tbP \rtimes \C^\times_{rot})$$ and in particular
$$\Gr_G\cong \tG^{\cO}_{\cK}/\tG_{\cO} \cong (\tG^{\cO}_{\cK} \rtimes \C^\times_{rot})/(\tG_{\cO} \rtimes \C^\times_{rot}).$$ Let $N$ be an algebraic representation of $\tG$ extending the representation of $G$ on $N$. 
\begin{definition}
	Define the {\em BFN space} of $(G,N,\bP,\bN)$ as
	$$\cR_{G,N,\bP,\bN}=\{([g], s)\in \Fl_\bP \times \bN|g.s\in \bN\}.$$
\end{definition}
\begin{remark}
	If $\bP=G_\cO, \bN=N_\cO$ we omit the subscripts $\bP, \bN$.
	We naturally have
	$$\cR_{G,N,\bP,\bN}\subset \cT_{G,N,\bP,\bN}:= G_\cK \times_{\bP} \bN \cong \{([g],s)\in \Fl_\bP \times N_{\cK}|g.s\in \bN\}$$
	
	The last isomorphism is given by the embedding $[g,s']\mapsto ([g],g^{-1}.s')$, see \cite[discussion on p.6]{BFN}, and note that we use $g^{-1}$ where they use $g$. We use these descriptions interchangeably. When $\bP=G_\cO, \bN = N_\cO$, $\cT_{G,N}$ has the modular interpretation 
	$$\cT_{G,N}\cong \{(P,\sigma,s)| P \text{ is a } G-\text{torsor on the formal disk } D, 
	\sigma: P|_{D^\times}\xrightarrow{\cong} G|_{D^\times}, s \in \Gamma(D, P\times_G N)\}.$$
	The locally closed sub-ind-scheme $\cR_{G,N}$ consists of those triples $(P,\sigma, s)$ where $\sigma(s)$ extends to a section over $D$. The versions with $\bP$ incorporate appropriate parabolic structure; i.e. we impose that $P$ have a $\bP$-reduction and require $s$ to be compatible with this reduction.
\end{remark}
\begin{theorem}[\cite{BFN}]
	\label{thm:BFN}
	There is a natural convolution product on $\cA_{G,N}:=H_*^{G_\cO}(\cR_{G,N})$ and $\cA^{\hbar}_{G,N}:=H_*^{G_\cO\rtimes \C^\times_{rot}}(\cR_{G,N})$, making them associative algebras with unit. Moreover, $\cA^{\hbar}_{G,N}$ is a flat deformation over $\C[\hbar]$ of $\cA_{G,N}$, which is commutative.
\end{theorem}
\begin{definition}
	We will call either of these algebras the {\em BFN algebra} or the (quantized) {\em Coulomb branch}.
\end{definition}
\begin{remark}
	The BFN algebra $\cA_{G,N}$ and its quantization have natural deformations given an extension as above. Namely, the homologies $\tilde{\cA}_{G,N}:=H_*^{\tG_\cO}(\cR_{G,N})$ and $\tilde{\cA}^{\hbar}_{G,N}:=H_*^{\tG_\cO\rtimes \C^\times_{rot}}(\cR_{G,N})$ have the structures of algebras that deform $\cA_{G,N}$ and $\cA_{G,N}^\hbar$, respectively, with $\tilde{\cA}^{\hbar}_{G,N}$ a filtered quantization of the commutative $\tilde{\cA}_{G,N}$. See \cite[Section 3(viii)]{BFN} for more details. This physically corresponds to turning on complex mass parameters for the flavor group $G_F$. In that context, one assumes that $G_F$ is a torus.
\end{remark}
\subsubsection{Parahoric versions}
A slight modification of the construction in \cite[Theorem 3.10]{BFN} gives:
\begin{theorem}
	There is a natural convolution product on $\cA_{G,N,\bP,\bN}:=H_*^{\bP}(\cR_{G,N,\bP,\bN})$ and $\cA^{\hbar}_{G,N,\bP,\bN}:=H_*^{\bP \rtimes \C^\times_{rot}}(\cR_{G,N,\bP,\bN})$, making them associative algebras with unit. Moreover, $\cA^{\hbar}_{G,N,\bP,\bN}$ is a filtered quantization of $\cA_{G,N,\bP,\bN}$.
\end{theorem}
\begin{remark}
	Similarly, one can define the flavor-deformed version $\tilde{\cA}_{G,N,\bP,\bN} = H_*^{\tbP}(\cR_{G,N,\bP,\bN})$ and its quantization $\tilde{\cA}_{G,N,\bP,\bN}^\hbar = H_*^{\tbP \rtimes \C^\times_{rot}}(\cR_{G,N,\bP,\bN})$. Note that unless $\bP=G_\cO$, the algebra 
	$\cA_{G,N,\bP,\bN}$ is in general {\em not} commutative. For example, $\tilde{\cA}_{G,N,\bP,N_\cO}$ is a matrix algebra (of size $\dim_\C G_\cO / \bP \times \dim_\C G_\cO / \bP$) over $\tilde{\cA}_{G,N}$, because the map $\cR_{G,N,\bP,\bN}\to \cR_{G,N}$ is a $G/P$-fibration. For more details, see \cite[Section 7.1]{lineops}.
\end{remark}
\begin{remark}[For the physically minded reader]
	The algebra $\cA_{G,N,\bP,\bN}$ encapsulates the algebra of local operators bound to a ($\tfrac12$-BPS) vortex line operator labeled by the algebraic data $\bP, \bN$. As described in \cite[Section 4.5]{lineops}, the choice of $\bP$ is a breaking of the gauge group in the vicinity of the line operator. The choice $\bN$ is related to a choice of superpotential (compatible with the choice of symmetry breaking) coupling the bulk degrees of freedom to the degrees of freedom on the line operator. Examples of such line operators have been used to obtain non-commutative resolutions of Coulomb branches \cite{BFN-lines} and played a central role in understanding of symplectic duality between Higgs and Coulomb branches \cite{WebKD}.
\end{remark}

\subsection{Convolution action of Coulomb branches on GASF}
Recall that we have defined the BFN space $\cR_{\bP,\bN}:=\cR_{G,N,\bP,\bN}$ of a representation $N$. We will also consider the infinite-rank vector bundle $$\cT_{\bP,\bN}:=\cT_{G,N,\bP,\bN}:=G_\cK\times_{\bP} \bN \to \Fl_\bP.$$ 
For $v \in N_\cK$, we define $$\cV^v_{\bN}:=(\tG_\cK^\cO \rtimes \C^\times_{rot}).v\cap \bN.$$ This is analogous to the {\em orbital varieties} in \cite{CG}, and is also called such by \cite{HKW} in the case $\bP=G_\cO$.
Note that on the level of closed points (which is what we are concerned with, since we only work with the reduced structure), it is clear that $\cV^v_{\bN}/(\tbP \rtimes \C^\times_{rot})=M_v^{\bP,\bN}/L_v$.

We now define the convolution action of $\tilde{\cA}^\hbar_{\bP,\bN}:= \tilde{\cA}^\hbar_{G,N,\bP,\bN}$, following \cite{BFN} and \cite{HKW} (which consider the case $\bP=G_\cO, \bN=N_\cO$). 
\begin{theorem}
	\label{thm:convolution}
	Suppose the stabilizer $L_v$ of $v$ is contained in $\tbP \rtimes \C^\times_{rot}$ and $M_v^{\bP,\bN}$ is ind-proper. Then there is an action of $\tilde{\cA}^{\hbar}_{\bP,\bN}$ on $H_*^{L_v}(M_v^{\bP,\bN})$.
\end{theorem}
\begin{proof}
	Note that there is a natural map 
	\begin{equation}
	\label{eq:projection}
	p: \tG_\cK^\cO \rtimes \C^\times_{rot} \times \bN \to \cT_{\bP,\bN} \times \bN
	\end{equation} given by $$(g,s)\mapsto ([g,s],s).$$ 
	
	Let $L_v$ be the stabilizer of $v$ in $\tG_\cK^\cO\rtimes \C^\times_{rot}$. If $X^{\bP,\bN}_v:=\{g\in \tG^{\cO}_{\cK}\rtimes \C^\times_{rot}|g^{-1}.v\in \bN\}$, there are two natural projections from $X^{\bP, \bN}_v$ to $M_v^{\bP,\bN}$ and $\cV_{\bP,\bN}^v$, which are $\tbP \rtimes \C^\times_{rot}$ and $L_v$-torsors, respectively.
	Taking the equivariant cohomology of the dualizing sheaves, we get
	$$H^{L_v}_*(M_{v}^{\bP,\bN}) \simeq H^{\tbP \rtimes \C^\times_{rot}}_*(\cV^v_{\bP,\bN}),$$
	where the left-hand side makes sense because $L_v$ is compact.   
	
	Consider the groupoid over $\bN$ given by $$\cP_{\bP,\bN}:=\{(g,s)\in \tG^{\cO}_{\cK}\rtimes \C^\times_{rot}\times \bN| g^{-1}.s\in \bN\}\xrightarrow{\pi_1:(g,s)\mapsto s} \bN.$$
	The source, target, unit, multiplication and inverse maps are defined as in \cite[Section 4.2]{HKW}.
	Note that there is another projection map $\pi_2$ to $\bN$ given by $(g,s)\mapsto g^{-1}.s$. Then consider $\cF^v_{\bP,\bN}:=\omega_{\cV_{\bP,\bN}^v}[-2\dim \tbP]$, which is an object in the $\tbP \rtimes \C^\times_{rot}$-equivariant derived category of $\bN$. Here we use, as everywhere in this paper, the grading shift like defined in \cite{BFN} and Appendix \ref{appendix}. We have a natural isomorphism $$\pi_1^*\cF^v_{\bP,\bN}\cong \pi_2^*\cF^v_{\bP,\bN}.$$
	
	By definition we have $p^{-1}(\cR_{\bP,\bN}\times \bN)=\cP_{\bP,\bN}$, and that $m\circ q=\pi_2$, 
	$\pi\circ j=\pi_1$, where $$\pi: \tG_{\cK}^{\cO}\rtimes \C^\times_{rot} \times \bN\to \bN$$ is the projection.

	Consider then the  following diagram:
	\begin{center}
\includegraphics{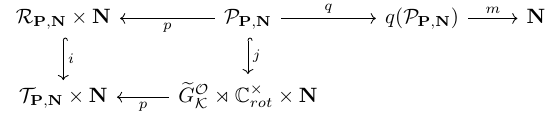}
		\label{eq:convolutiondiagram}
	\end{center}
	Here $p$ is as above, $q$ is quotient by the $\tbP \rtimes \C^\times_{rot}$-action $h.(g,s)=(gh^{-1},h.s)$ and $m$ is the multiplication map $[g,s]\mapsto g.s$. The composition $m \circ q$ is the above map $\pi_2: \cP_{\bP,\bN} \to \bN$.
	
	Using the ``restriction with support" map of Section \ref{sec:rws} (see \cite[Section 3(ii)]{BFN}) applied to the leftmost Cartesian square, and the map 
	$$p^*(\omega_{\cT_{\bP,\bN}}[-2\dim \bN]\boxtimes \cF^v_{\bP,\bN})\cong \omega_{\tG_{\cK}^\cO\rtimes \C^\times_{rot}}[-2\dim \tbP\rtimes \C^\times_{rot}]\boxtimes \cF^v_{\bP,\bN}$$
	we get a map (omitting the shifts for sake of readability)
	\begin{align}
	p^*:\, &H_{\tbP \rtimes \C^\times_{rot} \times \tbP \rtimes \C^\times_{rot}}^{-*}(\cR_{\bP,\bN}\times \bN, \omega_{\cR_{\bP,\bN}}\boxtimes \cF^v_{\bP,\bN})\nonumber\\ 
	& =H_*^{\tbP\rtimes \C^\times_{rot}}(\cR_{\bP,\bN})\otimes H_*^{\tbP\rtimes \C^\times_{rot}}(\cV^v_{\bN})\to H^*_{\tbP\rtimes \C^\times_{rot}\times \tbP\rtimes \C^\times_{rot}}(\cP_{\bP,\bN}, \pi_1^!\cF^v_{\bP,\bN}).
	\end{align}
	
	Since $\cF^v_{\bP,\bN}$ is a $\bP$-equivariant complex, we have
	$\pi_1^!\cF^v_{\bP,\bN}\cong \pi^!_2\cF^v_{\bP,\bN}$ and since $\pi_2=m\circ q$, we get 
	$$H^*_{\tbP\rtimes \C^\times_{rot}}(\cP_{\bP,\bN}, \pi_1^!\cF^v_{\bP,\bN})=H^*_{\tbP\rtimes \C^\times_{rot}}(q(\cP_{\bN}), m^!\cF^v_{\bP,\bN})$$
	
	Finally, $m$ is (ind-)proper because its fibers are closed subvarieties of a partial affine flag variety, so that using the adjunction $m_!m^!\to \id$ we get a map
	$$(m\circ q)_*: H_*^{\tbP \rtimes \C^\times_{rot}}(q(\cP_{\bP,\bN}),m^!\cF^v_{\bP,\bN})\to H_*^{\tbP \rtimes \C^\times_{rot}}(\cV^v_{\bP,\bN}).$$
	
	In particular, composing gives us an ``intersection pairing" $$\star:=(m\circ q)_*p^*:H^{\tbP \rtimes \C^\times_{rot}}_*(\cR_{\bP,\bN}) 
	\times H^{L_v}_*(M^{\bP,\bN}_v) \to H^{L_v}_*(M^{\bP,\bN}_v).$$ This is clearly bilinear over $\Q$. We prove the associativity in Lemma \ref{lemma:associativity} and the fact that the identity acts by $1$ in Lemma \ref{lemma:unit}.
\end{proof}

\subsubsection{The case of Hilbert schemes}
Specializing the construction of Theorem \ref{thm:convolution} to $N=\Ad \oplus V$ and $\bP=G_\cO, \bN=N_\cO$, the $L_v$-equivariant homology of the GASF $M_v$ admits an action of the spherical rational Cherednik algebra of $\fg\fl_n$. Similarly, for $\bP=\bI, \bN=\Lie(\bI)\oplus \cO^n$ we get an action of the (full) RCA of $\fg\fl_n$, as we now describe.
\begin{definition}
	\label{def:cherednikalg}
	The rational Cherednik algebra of $\fg\fl_n$ is the quotient algebra
	$$\overline{\cH}_n=\frac{\C[\hbar,m]\langle x_1,\ldots,x_n,y_1,\ldots,y_n\rangle\rtimes\C\Sn}{\sim}$$
	where $\sim$ consists of the relations $[x_i,x_j]=[y_i,y_j]=0$ for all $i, j$, and 
	$$[y_i,x_j]=\begin{cases}
	-\hbar+m\sum_{k\neq i} (i\,k) & \text{ if } i=j,\\
	-m(i\,j) & \text{ if } i \neq j.
	\end{cases}$$
	
	The {\em spherical subalgebra} is defined as $e\overline{\cH}_ne$ where $e=\tfrac{1}{n!}\sum_{w\in \Sn} w$. We often refer to the spherical subalgebra simply as the spherical rational Cherednik algebra of $\fg\fl_n$. 
\end{definition}
\begin{remark}
	To match with the conventions in most other sources, we should specialize $\hbar \to -1$. Indeed, it is the specialized algebra which will act on the equivariant homology as in Theorem \ref{thm:glnirrep} and for example \cite{OY}.
\end{remark}
We record the following theorems of Kodera-Nakajima and Braverman-Etingof-Finkelberg \cite{KN, BEF} (see also \cite{WebDO,WebLP}).
\begin{theorem}[\cite{KN}]
	\label{thm:rca}
	For $G = GL_n, N = \Ad\oplus V,$ the quantized BFN algebra $\tilde{\cA}^{\hbar}_{G,N}$ is isomorphic to the {\em spherical rational Cherednik algebra} of $\fg\fl_n$.
\end{theorem}
\begin{theorem}[\cite{BEF}]
	For $G = GL_n, N = \Ad\oplus V, \bP=\bI, \bN=\Lie(\bI)\oplus \cO^n$, the quantized BFN algebra $\tilde{\cA}^{\hbar}_{G,N,\bP,\bN}$ is isomorphic to the {\em rational Cherednik algebra} of $\fg\fl_n$.
\end{theorem}
\begin{remark}
	The extended group $\tG$ in the above theorems is simply $G \times G_F$ where $G_F=\C^\times_{dil}$ acts by scaling $\Ad$ with weight $1$ and $V$ with weight 0.
\end{remark}

In the situation of Theorem \ref{thm:maincomparison} we get 
\begin{corollary}
	The spherical rational Cherednik algebra $e\overline{\cH}_ne$ of $\fg\fl_n$ acts on  
	$H_*^{L_v}(\Hilb^\bullet(\hC))$ where $L_v$ is the stabilizer in $\tG_\cK^\cO\rtimes \C^\times_{rot}$ of $v\in \Ad(\cK)\oplus \cK^n$ associated to $\hC$ as in Theorem \ref{thm:maincomparison}.
\end{corollary}
\begin{corollary}
	\label{cor:fullrcaaction}
	The rational Cherednik algebra $\overline{\cH}_n$ of $\fg\fl_n$ acts on $H_*^{L_v}(\PHilb^{\bullet}(\hC))$.
\end{corollary}
\begin{remark}
	The action in Corollary \ref{cor:fullrcaaction} coincides by \cite[Section 7]{WebDO} with that studied in \cite{GSV}. Both papers use a different set of generators than us, and we compare their construction to ours in Section \ref{sec:gsv}.
\end{remark}

\subsection{Comparison of the convolution action to an action by correspondences}
\label{sec:corr}
For many of our results, in particular Theorem \ref{thm:glnirrep}, we will need to compare the convolution action from Theorem \ref{thm:convolution} to another action by correspondences. We will do this again in greater generality than needed for the rest of the paper. In particular, we make rigorous expectations from \cite{BDG} and \cite{BDGH}.

\begin{definition}
	\label{def:raviolo}
	Define the {\em raviolo space/Hecke stack} for $v$
	which has $\C$-points given by $$\cR^{v}_{\bP,\bN}(\C) = \{(s_2, g, s_1)\in \cV^v_{\bN}\times (\tG^\cO_\cK\rtimes \C^\times_{rot})  \times \cV^v_{\bN} | g.s_1 = s_2\}/\tbP \rtimes \C^\times_{rot}.$$
	Explicitly, the element $g' \in \tbP \rtimes \C^\times_{rot}$ acts as $g'.(s_2, g, s_1) = (s_2, g g'^{-1}, g' s_1)$.
\end{definition}
\begin{definition}
	\label{def:traviolo}
	Define also $$\cT^{v}_{\bP,\bN}(\C) = \{(s_2, g, s_1)\in \cW^v\times (\tG^\cO_\cK\rtimes \C^\times_{rot})  \times \cV^v_{\bN} | g.s_1=s_2\}/\tbP \rtimes \C^\times_{rot},$$ where $\cW^v:=(\tG^\cO_\cK\rtimes \C^\times_{rot}).v\subset N(\cK)$.
\end{definition}

Next, note that $\cR^{v}_{\bP,\bN}$ is a locally closed sub-ind-variety of $\cR_{\bP,\bN}$ via $[s_2,g,s_1]\mapsto [g,s_1]$. The space $\cR_{\bP,\bN}$ has a natural stratification pulled back from the Schubert stratification of $\Fl_\bP=\bigsqcup_{w\in W_\bP\backslash W^{aff}/W_{\bP}} \bP w \bP$ and therefore $\cR_{\bP,\bN}^v$ inherits a stratification $\cR^{v}_{\bP,\bN, w} := \cR^{v}_{\bP,\bN} \cap \cR_{\bP,\bN, w}$ (ditto for $\cT^v_{\bP,\bN}, \cT_{\bP,\bN}$). Here $w\in W_{\bP}\backslash W^{aff}/W_\bP$ is a double  coset for the extended affine Weyl group of $G$. We also have maps
\begin{center}
\includegraphics{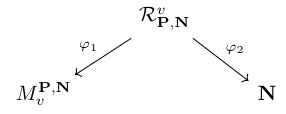}
\end{center}
where $\varphi_1$ is the $\tbP \rtimes \C^\times_{rot}$-
equivariant projection map $$\varphi_1: [s_2,g,s_1]\mapsto [s_1],$$ whose restriction to 
$\cR^{v}_{\bP, \bN, w}$ is smooth (because it is a base change from the smooth morphism $
\cR_{\bP,\bN,w}\to \Fl_\bP$), and $\varphi_2$ is another proper equivariant projection given by $$\varphi_2: [s_2,g,s_1]\mapsto s_2,$$ whose image is naturally identified with $\cV^v_{\bN}$.

The map in Equation \eqref{eq:projection} restricts to \begin{eqnarray}
\label{eq:ravioloproj}
p: p^{-1}(\cR_{\bP, \bN}\times \cV_{\bN}^v)& \to \cR_{\bP, \bN}^v\times \cV_{\bN}^v \\
p: \tG^\cO_\cK\rtimes \C^\times_{rot}\times \cV_{\bN}^v &\to \cT_{\bP,\bN}^v\times \cW^v\end{eqnarray}
and $q(p^{-1}(\cR_\bP\times \cV_\bP^v))\cong \cR_\bP^v$ by the right quotient. Note that when the stabilizer of $v$ is trivial, we have $p^{-1}(\cR\times \cV^v)\cong \cV^v\times \cV^v$ by $(g,s)\mapsto (s,g.s)$. Our goal is to interpret the ``push-pull" maps in equivariant cohomology of $\cV^v$ giving rise to the action.

Note that $q_*p^*$, where $p^*$ is defined in Theorem \ref{thm:convolution}, defines a map $$H_*^{\tbP\rtimes \C^\times}(\cR_{\bP,\bN}^v\times \cV_{\bN}^v)\to H^*_{\tbP\rtimes\C^\times_{rot}}(\cP_{\bN},\pi_1^!\cF^v_{\bP, \bN})=
H^{\tbP\rtimes\C^\times_{rot}}_*(\cR_{\bP,\bN}^v).$$

Given a class $\beta \in \cA_{\bP, \bN,\leq w}^\hbar$ and $\alpha\in H_*^{L_v}(M^{\bP, \bN}_v)\cong H_*^{\tbP\rtimes \C^\times}(\cV_{\bN}^v)$ 
we have that $q_*p^*(\beta \otimes \alpha)$ is identified with the restriction of the map $q_*p^*$ to $\cR^v_{\bP, \bN, \leq w}\times \cV^v_{\bN}$. 

In particular, whenever $\cR_{\bP, \bN,w}$ is closed (for example, minuscule in the $G_\cO$-case) there is a well-defined class $[\cR_{\bP, \bN, w}]\in\cA_{\bP, \bN,\leq w}^\hbar$. By smoothness of the maps in Eq. \eqref{eq:ravioloproj} and the natural inclusion $\cR_{\bP, \bN}\to \cT_{\bP, \bN}$ we may then use the ``classical" refined pullback map as in \cite{Fulton} to compute $q_*p^*([\cR_{\bP, \bN, w}]\otimes \alpha)$ given good enough understanding of $\cR^v_{\bP, \bN}$ and how it sits in $\cT_{\bP, \bN}^v$.
Moreover,   $m_*: H_*^{\tbP\rtimes\C^\times_{rot}}(\cR_{\bP, \bN}^v)\to H^{L_v}_*(M_v^{\bP, \bN})$ as given as in Theorem \ref{thm:convolution} is identified with $\varphi_{2,*}$.
In Section \ref{sec:fixedpts} we will see that it is possible to compute $(m\circ q)_*p^*$ using this interpretation in the abelian setting, which enormously simplifies computations.

\subsubsection{The case of Hilbert schemes} 
Suppose now $\lambda=\omega_1=(1,0,\ldots, 0)$ is the first minuscule coweight of $GL_n$ and we are in the setting of Theorem \ref{thm:maincomparison}. We will compare our results to those of \cite{Kiv1}. Denote for ease of notation $\cR:=\cR_{GL_n,\Ad\oplus V}$. We have the following proposition.
\begin{proposition}
We have an isomorphism of stacks 
$$[\tG_\cO\rtimes \C^\times \backslash \cR_{\leq \lambda}^v]\cong [L_v\backslash \Hilb^{\bullet,\bullet+1}(C)]$$ where $\Hilb^{\bullet,\bullet+1}(C)$ is the flag Hilbert scheme of $C$ consisting of finite colength ideals $I_1\subset I_2\subset \cO_C$ with $\dim I_2/I_1=1$. Therefore, $H_*^{\tG_\cO\rtimes \C^\times}(\cR_{\leq \lambda}^v)\cong 
H^{L_v}_*(\Hilb^{\bullet,\bullet+1}(C))
$. While the definitions of most of the objects below make sense in general, in order to apply the deformation theory results from \cite{Kiv1,She12}, we need to assume $C$ is the germ of a {\em reduced} (but possibly reducible) curve.

\end{proposition}
\begin{proof}
Recall that $\Gr_{GL_n}^{\leq \omega_1}$ can be identified with lattices $\Lambda \subset\cO^n$ such that $\ell(\cO^n/\Lambda)=1$. Using this, we check that the points of $[\tG_\cO\rtimes \C^\times \backslash \cR_{\leq \omega_1}^v]$ correspond to points of $\Hilb^{\bullet,\bullet+1}(C)$ up to the stabilizer action on $v$.
\end{proof}

Recall from the proof of \cite[Theorem 1.1]{Kiv1} the diagram 

\begin{equation}
	\label{fig:hilbdiagram}
	\begin{aligned}
\includegraphics{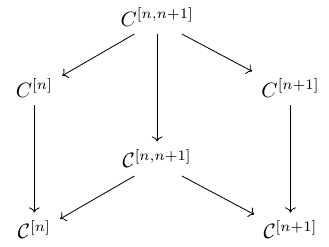}
	\end{aligned}
\end{equation}
Where the spaces in the lower row correspond to the total families of (flag) Hilbert schemes associated to a versal deformation of $C$. We will now use this diagram for $C$ a germ of a curve singularity 
(as in the rest of the paper).

\begin{theorem}
The map $\varphi^*_1=q_*p^*$ can be identified with the refined pullback map associated to the left square in Diagram \eqref{fig:hilbdiagram} coming from the embedding $C\hookrightarrow \cC$ (the resulting pullback is also denoted ``$p$" in \cite{Kiv1}) Similarly, if $\lambda=(0,\ldots, -1)$, we recover the refined pullback associated to the right square (called ``$q^*$" in {\em loc. cit.}).
\end{theorem}
\begin{remark}
Note that this only covers part of the main result of \cite{Kiv1}. In particular, the other maps defined in {\em loc. cit.} use global (projective) curves in their definition.
\end{remark}
\begin{proof}
The affine Grassmannian of $\GL_n$ is the increasing union of the projective varieties 
$$\Gr_{\GL_n}^d:=\{\Lambda\subset \cK^n|t^{d}\cO^n\subset \Lambda \subset t^{-d}\cO^n\}.$$
It is clear that $M_v$ as in Theorem \ref{thm:maincomparison} corresponding to the germ of a curve $\hC$ has $M_v^{m}:=\bigsqcup_{i=0}^m \Hilb^i(\hC)\subset M_v$ contained in $\Gr_{\GL_n}^d$ for all $m$ and some $d$ depending on $m$.

Let moreover $N_d:=N(\cO)/t^d N(\cO)$ and $\cV^v_d$ be the image of $\cV^v$ in the quotient by $t^d$. Let also $\cR^d:= \{[g,s]\in \Gr_G^d\times^{G(\cO)/t^d} N_d| g^{-1}. s\in N_d\}$. Then $\cR$ is  the colimit of $\cR^d$ for the inclusions coming from $\Gr_G^d\hookrightarrow \Gr_G^{d+1}$, in particular the equivariant Borel-Moore homology is the corresponding colimit.

Choose $d\gg 0$ and some open neighborhood $U$ of $v\in N_d$. Then choosing some transversal slice $S$ to $\cV^v_d$, we locally have 
$\cV^v_d\times S\cong U$, as in \cite[Proposition 6.5.13]{CG}.  In particular, if we let $\varphi: \cR^d\to N_d$ be the projection, and
$$\Sigma:=\varphi^{-1}(\cV^v_d), \; \Sigma_{U}:=\varphi^{-1}(U\cap \cV^v_d)$$ then $$\Sigma_U\cong (\cV^v\cap U)\times M_v^d.$$

Consider the inclusion $\cV^v_d\cap  U\hookrightarrow U$. The map $\varphi^{-1}(U)\to U$ is smooth, so we get a refined pullback map \cite{Fulton}
$$\varphi^*: H_*^{\tG(\cO)/t^d\rtimes \C^\times}(\cV^v_d\cap U)\to H_*^{\tG(\cO)/t^d\rtimes \C^\times}(\Sigma_U).$$ We will in fact abuse notation and denote by $\varphi^*$ the composition of this map and the pushforward
$$H_*^{\tG(\cO)/t^d\rtimes \C^\times}(\Sigma_U)\to  H_*^{\tG(\cO)/t^d\rtimes \C^\times}(\Sigma).$$

Possibly further increasing $d$ and throwing away some high codimension subset of $U$ not containing $v$, we claim that it is possible to identify $\varphi^{-1}(U)\to U$ with 
the family of Hilbert schemes of $0,1,\ldots, d$ points (i.e. the union thereof). Indeed, the fibers at $v=(\gamma,e)\in U$ of the map $\varphi$ are by the proof of Theorem \ref{thm:maincomparison} identified with Hilbert schemes of (at most $d$, since we truncate by $t^d$) points on the curve $\Spec \C[[t]][\gamma]$, whenever $e$ is a cyclic vector for $v$. Since having a cyclic vector is an open condition, we get a family of Hilbert schemes of points on some curves.

Since $N_d$ is the space of {\em all} matrices and vectors in $\cO/t^d$, the associated family of (germs of) planar curves is versal for large enough $U$. In particular, the characteristic polynomial map $\Ad(\cO)\to \C[[t,x]]$ restricted to the natural section $\Ad(\cO/t^d)\to \Ad(\cO)$ contains all degree $\deg(f)$ polynomials, and it is a standard result that this ensures versality, see e.g. \cite[Proposition 17]{She12}.

In turn, this implies the associated total spaces of the Hilbert schemes of points and the one-step flag Hilbert schemes of $C$ are smooth, by \cite[Proposition 17]{She12} and \cite[Proposition 2.6]{Kiv1} .

Further restricting $\varphi$ to $\varphi^{-1}(U)\cap \cR^{\leq\lambda}_d$ for the cocharacter $\lambda=(1,\ldots, 0)$
identifies the refined intersection map $p^*$ for the inclusion $v\hookrightarrow U$ in \cite[Definition 3.4]{Kiv1} with $\varphi^*_{\leq \lambda}$. The other case is similar.
\end{proof}
In particular, this gives an interpretation of one of the Weyl algebras appearing in \cite{Kiv1,Ren}. Recall that in {\em loc.cit.} it was shown that a certain subalgebra of the Weyl algebra of $\A^{2\cdot d}$ (where $d$ is the number of irreducible components of $C$) containing $\text{Weyl}(\A^1)\otimes \text{Weyl}(\A^1) \cong \text{Weyl}(\A^2)$ acts on the homologies of the Hilbert schemes of points in question. The other $\text{Weyl}(\A^1)$ in the algebra has to do with the Hilbert schemes of global curves and cannot be defined in our setting. Indeed, the definition of this other Weyl algebra  as a subalgebra of the one constructed in {\em loc.cit.} depends on the number of components of the curve, whereas our Cherednik algebra depends on the {\em degree} of the curve.
\begin{remark}
	It is remarkable to note that the convolution action works on the level of punctual Hilbert schemes directly. In \cite{Kiv1} and \cite{Ren}, one of the main points is to define convolution maps for the Hilbert schemes of (locally planar) singular curves using refined intersection products, which are constructed by deforming the singularities as we saw above. The role of the deformation in our context is played by considering the infinite-dimensional ind-variety $\cV^v$ in place of $M_v$. Note also that the ``restriction with supports" map is a refined intersection product in the case of a regular embedding, while here we use a rather special form of the map $p$, which is very far from anything like a regular embedding, but rather a principal bundle. 
\end{remark}

\subsection{Localization to fixed points}
\label{sec:fixedpts}
Let us analyze the construction of Theorem \ref{thm:convolution} first in the case $G=T$ is a torus. In this case, $\cR_T$ is a collection of (infinite rank) vector bundles over a discrete set $\Gr_T\cong X^*(T)$, of finite codimension in $\cT$.
Its complex points are $$\cR_T(\C) = \{(g,s)\in \tT^\cO_\cK\rtimes \C^\times_{rot} \times N(\cO) : g^{-1}.s \in N_\cO\}/\tT_\cO\rtimes \C^\times_{rot},$$ and the map $\pi_T: \cR_T \to \Gr_T$ given by forgetting $s$. The map $$\tT_\cK^\cO\times N_\cO\to \cT_T\times N_\cO$$ is simply many copies of the quotient map $$\C((t))^\times\to \C((t))^\times/\C[[t]]^\times.$$

Fix now $G$ reductive and $T$ a maximal torus in it. We may think of $\cR_T$ as an ``abelianized" BFN space for $G$, as it also admits an inclusion map $\iota: \cR_T \hookrightarrow \cR$ via inclusion of $\Gr_T\hookrightarrow \Gr_G$. The space $\cR_T$ has a natural convolution product and it admits a natural action of the Weyl group $W$. By Lemma 5.10 of \cite{BFN} there is an algebra homomorphism $(\iota_\cR^W)_*:(\cA_T^\hbar)^W \to\cA^\hbar$ coming from the inclusion $\iota_\cR: \cR_T \hookrightarrow \cR.$ We call $\cA_T^\hbar$ the ``abelianized" BFN algebra. This construction generalizes to the flavor deformed algebras $(\tilde{\cA}_T^\hbar)^W \to\tilde{\cA}^\hbar$, where $\tilde{\cA}^\hbar_{T}:= H^{\tT_\cO\rtimes \C^\times_{rot}}_*(\cR_T)$.

Consider $\cR^{v}_T = \cR_T \cap \cR^{v}$.
By definition of the generalized affine Springer fiber for $v$, where we consider $N$ as a representation of $T\subset G$, we see that $\cR^v_T$ is the Hecke stack associated to the datum $(T,N,v)$.
Using the convolution action of Theorem \ref{thm:convolution} for $(T,N)$, we get an action of $\tilde{\cA}_T^\hbar$ on $H^{L_{T,v}}_*(M_{v,T})$ where $L_{T,v}$ is the stabilizer of $v$ in $T$.

We can now try to compare the two actions.

\begin{proposition}
	\label{prop:torusloc}
	Suppose $\Gr_G$ has isolated fixed points under the stabilizer $L_v\subset \tG_\cK^\cO\rtimes \C^\times_{rot}$ of $v$ and that $L_v$ is contained in $\tT_\cO \rtimes \C^\times_{rot}$. Suppose also $M_v$ is equivariantly formal for $L_v$. Then 
	\begin{enumerate}
		\item $M_{v,T}=M_v^{L_v}$ 
		\item $(\iota_{M_v})_*: H_*^{L_v}(M_v^{L_v})\to H_*^{L_v}(M_v)$ becomes an isomorphism after inverting countably many characters of $L$.
		\item $(\iota_{M_v})_*$  intertwines the actions of $(\tilde{\cA}^\hbar_{T,N})^W$ and $\tilde{\cA}^\hbar_{G,N}$.
	\end{enumerate}
\end{proposition}
\begin{proof}
	The first assertion follows from the fact that the $L_v$-fixed points are contained in the $L_v$-fixed points on the affine Grassmannian, which we assume are also the $T$-fixed points. On the other hand, these fixed points for $T$ are topologically a discrete set of points coinciding with $\Gr_T$. 
	The second assertion is the GKM (or in some contexts, Atiyah-Bott) localization theorem. 
	
	Consider the following diagram:
	\begin{center}
\includegraphics{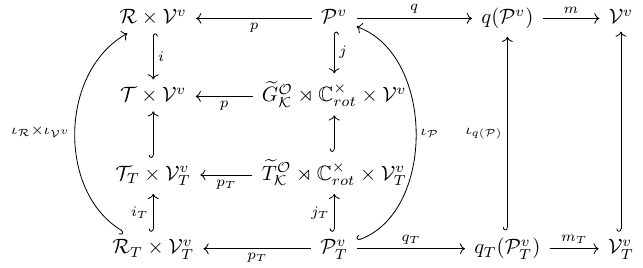}
	\end{center}
	Here $i, j, p, q, m$ are as before, and the versions with subscript $T$ are the corresponding maps for $T\subset G$. The inclusions $\iota_{?}$ come from the maps $T\hookrightarrow G, \Gr_T\hookrightarrow \Gr_G$ and variations. The space $\cP^v$ is defined as $\cP^v:=p^{-1}(\cR\times \cV^v)$ and $\cP^v_T$ by replacing $G$ with $T$.
	
	Note that the upper and lower squares on the left tower of squares are clearly Cartesian. We claim that the middle one is so too. By definition the fiber product 
	$$(\tG_{\cK}^\cO\rtimes\C^\times_{rot} \times \cV^v)\times_{\cT \times \cV^v}(\cT_T \times \cV^v_T)$$
	consists of $(g,s',[t,s])$ so that $[g,s']=[t,s]$ and $s=s'$. In particular, there is some $g'\in L_v$ such that $gg'=t$. But since $L_v$ is contained in $\tT_\cO\rtimes \C^\times_{rot}$, we must have $g\in \tT_\cK^\cO\rtimes \C^\times_{rot}$. So every square in the tower is Cartesian. Note that this is not true without our assumptions (take for example $N=0, v=0$).
	
	Let $\cF=\omega_{\cV^v}[-2\dim \tG_\cO]$ and
	$\cF_T=\omega_{\cV^v_T}[-2\dim \tT_\cO]$. 
	Let $\iota_{\cV^v}: \cV^v_T\hookrightarrow \cV^v$. Then $\iota_{\cV^v_T}^!\cF=\cF_T$ by functoriality of the upper shriek and the definition of the dualizing complex.

	Let then $$r\otimes \alpha \in H_*^{\tG_\cO\rtimes \C^\times_{rot}}(\cR)\otimes H^{L_v}_*(M_v)\cong
	H^{-*}_{\tG_\cO\rtimes\C^\times_{rot}\times\tG_\cO\rtimes\C^\times_{rot}}(\cR\times \cV^v,\omega_\cR\boxtimes \cF).$$ 
	By Lemma 5.10. of \cite{BFN}, the pushforward map 
	$$(\iota_\cR)^W_*: (\tilde{\cA}^{\hbar}_T)^W\to \tilde{\cA}^\hbar$$ given by taking the $W$-invariants of the $\tT_\cO\rtimes \C^\times_{rot}$-equivariant pushforward becomes an isomorphism after localizing at countably many characters of $\tT\times \C^\times_{rot}$. By parts (1) and (2), 
	$$(\iota_{\cV^v})_*: H_*^{\tT_\cO\rtimes \C^\times_{rot}}(\cV^v_T)\cong H^{L_v}_*(M^{L_v}_v)\to 
	H^{L_v}_*(M_v)\cong H_*^{\tG_\cO\rtimes \C^\times_{rot}}(\cV^v)$$ also becomes an isomorphism after localizing at countably many characters of $L_v$.
	
	If we define moreover $$\iota_*:=(\iota_\cR)^W_*\otimes (\iota_{\cV^v})_*$$ and work in this localization,
	the intertwining property we need to show becomes
	$$\iota_*(m_T\circ q_T)_* p_T^*((\iota_*)^{-1}(r \otimes\alpha))=
	(m\circ q)_*p^*(r\otimes \alpha).$$
	Define $$A:=\omega_{\cT}[-2\dim N_\cO]\boxtimes \cF, A_T:=\omega_{\cT_T}[-2\dim N_\cO]\boxtimes \cF_T$$
	and $$B:=\omega_{\tG_{\cK}^\cO\rtimes \C^\times_{rot}}[-2\dim\tG_\cO\rtimes\C^\times_{rot}]\boxtimes \cF, B_T:=\omega_{\tT_{\cK}^\cO\rtimes \C^\times_{rot}}[-2\dim\tT_\cO\rtimes\C^\times_{rot}]\boxtimes \cF_T$$
	The restriction with support map $p^*$ from Theorem \ref{thm:convolution} and Definition \ref{def:rws} is (the induced map in hypercohomology of) the composition
	$$i^!A\to i^!p_*p^*A=p_*j^!p^*A\to p_*j^!B.$$
	Similarly we have 
	$$i^!_TA_T= (\iota_\cR\times \iota_{\cV^v})^!i^! A\to i^!_Tp_{T*}p_T^*A_T\to p_{T*}j_T^!B_T$$
	Using proper base change, we rewrite this as 
	$$(\iota_\cR\times \iota_{\cV^v})^!i^!A\to i^!_Tp_{T*}p_T^*A_T=(\iota_\cR\times \iota_{\cV^v})^!i^!p_*p^*A\to p_{T*}j_T^!B_T=(\iota_\cR\times \iota_{\cV^v})^!p_*j^!B.$$
	Passing to $\tT_\cO\rtimes \C^\times$-equivariant hypercohomology, we get that the square
	\begin{center}
\includegraphics{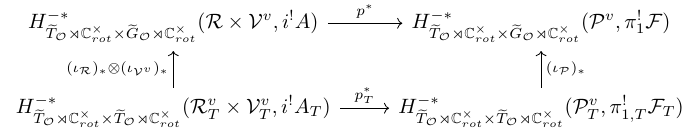}
	\end{center}
	commutes. Now taking $W$-invariants on the $\cR$-factor everywhere and passing to the localization where the left column becomes an isomorphism, we get
	$$p_T^*((\iota_*)^{-1}(r \otimes\alpha))=
	(\iota_{\cP*})^{-1}p^*(r \otimes \alpha).$$
	Since the right large square is also Cartesian and $\iota_\cP$ is a closed embedding, using proper base change  once more we get 
	$$\iota_*(m_T\circ q_T)_*(\iota_{\cP})^{-1}_*p^*(r \otimes \alpha)=(m\circ q)_*p^*(r \otimes \alpha).$$
\end{proof}
\begin{remark}
	Parts (1) and (2) of the above Proposition were also obtained in \cite[Theorem 5.13]{HKW}.
\end{remark}
\begin{remark}
	While it is natural to anticipate similar localization results for the parahoric cases, we do not know how these work due to a lack of an obvious replacement for the map $\Gr_T\hookrightarrow \Gr_G$ respecting the convolution structure in the case of other partial affine flag varieties.
\end{remark}

\subsubsection{Computations in the spherical case}

Let $\Gr_{G}^\lambda$ be the $G_\cO$-orbit of $t^\lambda \in \Gr_{G}$ and set $\cR^{\leq\lambda} = \cR \cap \pi^{-1}(\overline{\Gr_G^{\lambda}})$, where $\pi:\cR \to \Gr_{G}$ is the projection forgetting $N_\cO$. In what follows we will determine the action of various classes in $\tilde{\cA}^{\hbar}$ by means of two-fold fixed-point localization. Recall that there are commutative subalgebras $H^*_{\tG\times \C^\times_{rot}}(pt) \hookrightarrow \tilde{\cA}^\hbar$ and $H^*_{\tT \times \C^\times_{rot}}(pt) \hookrightarrow \tilde{\cA}_T^\hbar$. Denote the equivariant parameters for the torus $\tT \times \C^\times_{rot}$ collectively by $\varphi$ (for $T$), $m$ (for $G_F$) and $\hbar$ (for $\C^\times_{rot}$). 

Let $[t^\lambda]$ denote the fundamental class of $\cR_T \cap p_T^{-1}(\overline{\Gr_T^{\lambda}}),$ often called an ``abelianized monopole" \cite{BDG,BDGH}.
For $\lambda$ dominant with $\Gr_{G}^\lambda$ closed (minuscule) we can then write the following localization formula, c.f. \cite[Proposition 6.6.]{BFN}: \begin{equation}
\label{eq:abelianmonompole}
[\cR^{\leq \lambda}] = \iota_*\bigg(\sum\limits_{w \in W/W_\lambda}\frac{[t^{w.\lambda}]}{e(T_{w.\lambda} \Gr_G^{\leq \lambda})}\bigg),\end{equation} where $W_\lambda$ is the stabilizer of $\lambda$ in the Weyl group $W$.%
\footnote{Since the $\varphi$ do not commute with $[t^\lambda]$, we take the convention that the denominator is to the right of the numerator in writing this formula.}
The unit of the algebra $\tilde{\cA}^\hbar$ is $1:=[\cR^{\leq 0}]$. Other generators of $\tilde{\cA}^\hbar$ can be constructed by including a $W_\lambda$-invariant function $f(\varphi, m, \hbar)$ to the numerator of this expression: 
\begin{equation}
\label{eq:dressedmonopole}
[\cR^{\leq \lambda}][f] = \iota_*\bigg(\sum\limits_{w \in W/W_\lambda}\frac{(w.f) [t^{w.\lambda}]}{e(T_{w.\lambda} \Gr_G^{\leq \lambda})}\bigg)\end{equation}
These are called ``dressed" monopole operators, which are known to generate $\tilde{\cA}^\hbar$ in the cases we are interested in; see \emph{e.g.} \cite[Theorem A.7]{BFNSlices} and more generally \cite[Theorem 3.7]{WeekesGens}.

\begin{remark}
	More precisely, Theorem 3.7 of \cite{WeekesGens} says that the $[\cR_{\leq \lambda}][f]$ with {\em minuscule} $\lambda$ and a slightly smaller collection of $f$'s generate $\tilde{\cA}_{G,N}^\hbar$ for any quiver gauge theory; the quiver in the present case is a Jordan quiver with a framing node of rank 1 and was also described in Appendix A(iii) of \cite{BFNSlices}.
\end{remark}
\begin{remark}
	The terminology ``dressed monopole" has its origins in the physics literature, in our context they appear for example  in \cite{CHZ}. These operators also appear as the dimensional reduction of the four-dimensional mixed Wilson-'t Hooft operators of \cite{Kap}.
\end{remark}
Assume the hypothesis of Proposition \ref{prop:torusloc} and, moreover, that the map $L_v \to G_F \times \C^\times_{rot}$ is injective. Thus, the action of $H^*_{\tilde{T}\times \C^\times}(pt)$ factors through the action of $H^*_{L_v}(pt)$ \cite{HKW}.

We saw above that, if the induced action of the stabilizer subgroup $L_v$ on $\Gr_G$ has isolated fixed points and is contained in $\tT_{\cO} \rtimes \C^\times_{rot}$, the fixed points of the $L_v$ action on $M_v$ can be identified with elements of $\Gr_T \subset \Gr_G$. Of course, if $[p] \in M_v$ is a fixed point of the $L_v$ action and is represented by $p \in T_\cK$, $p$ is only fixed by $L_v$ up to the action of $T_\cO$, \emph{i.e.} we need to compensate the action of $L_v$ on $T_\cK$ with the action of $T_\cO$. The requirement that $L_v \to G_F \times \C^\times_{rot}$ is injective implies that there is a unique such compensating transformation, hence the action of $H^*_T(pt) \subset H^*_{\tT \times \C^\times_{rot}}(pt)$ on the fixed point class $\ket{p}$ is uniquely determined by the action of $H^*_{T_F \times \C^\times}(pt)$ on $\ket{p}$. We write $\varphi \ket{p} = \varphi(p) \ket{p}$. The action of $H^*_{T_F \times \C^\times}(pt)$ is then determined by the injection $L_v \to G_F \times \C^\times_{rot}$, which imposes $\textrm{rank } G_F + 1 - \textrm{rank } L_v$ linear relations on the $m \ket{p}, \hbar \ket{p}$. This is the source of the specialization discussed earlier.
\begin{remark}
	The bra-ket notation used to denote the fixed point classes $\ket{p}$ is used due to the realization of these classes as vectors in the supersymmetric Hilbert space in the gauge theory setup. It is worth noting that this isn't an honest Hilbert space, as the space of states in the twisted theory need not have a positive-definite inner product. Nonetheless, there is a natural symmetric, non-degenerate pairing of classes, c.f. \cite[Section 3.3]{BDGHK}.
\end{remark}
\begin{lemma}
	\label{lemma:monopoleAction}
	Assume that $M_v$ has isolated fixed points under the action of $L_v \subset \tT \times \C^\times_{rot}$ and that the map $L_v \to G_F \times \C^\times_{rot}$ is injective.
	
	For $\lambda$ a minuscule cocharacter and $f(\varphi, m, \hbar)$ a $W_\lambda$-invariant function we have
	\begin{equation}\label{eq:monopoleAction}
	[\cR_{\leq \lambda}][f] \ket{p} = \sum\limits_{w \in W/W_\lambda}\frac{\big(w.f(\varphi(t^{w.\lambda} p),m,\hbar)\big)  e(E_{p,w.\lambda})}{e(T_{w.\lambda} \Gr_G^{\leq\lambda})}\ket{t^{w.\lambda} p},
	\end{equation}
	where $E_{p,\nu}$ is an excess intersection factor. The denominator in this formula should be understood as replacing $\varphi$ in the polynomials $e(T_{w.\lambda} \Gr_G^{\leq \lambda})$ with $\varphi(p)$.
\end{lemma}
\begin{proof}
	By the previous Proposition we only need to compute this inside $H^{L_v}_*(M_v^{L_v})\otimes \C(\fl)$.
	Let $\ket{p}$ be (the inclusion of) the fundamental class of a fixed point in $M_v \subset N(\cO)/G(\cO)$. The subalgebra $H^*_T(pt)=\C[\ft] \subset \tilde{\cA}_T^\hbar$ acts as $\varphi_a \ket{p} = \varphi_a(p)\ket{p}.$
	Since $\pi_T^{-1}(\Gr_T^{\lambda})$ is a vector bundle over a point, using the excess intersection formula for the refined pullback $p^*$ (see Fulton \cite[Section 6.3]{Fulton}) we have  $$[t^\lambda] \ket{p} =
	(m\circ q)_* p^* ([t^\lambda]\otimes [p])=
	e(E_{p,\lambda}\big)\ket{t^\lambda p}.$$ As a vector space over $\C$, $E_{p,\lambda}$ can be expressed as
	$$E_{p,\lambda} \simeq N(\cO)/ (N(\cO)\cap t^{-\lambda}N(\cO)).$$ The equivariant structure of this vector space is determined by $\lambda$ and $p$; $E_{p, \lambda}$ should be thought of as a quotient of tangent spaces at $(t^\lambda p, t^\lambda, p) \in \cT_T$. A straightforward computation shows that $$e(E_{p,\lambda}\big) = \prod\limits_{\tilde{\mu} \textrm{ s.t. } \langle \mu, \lambda\rangle < 0}[\langle\tilde{\mu},\varphi(p)+m\rangle]^{\langle \mu, \lambda\rangle},$$ where the product runs over the $\tG$ weights $\tilde{\mu}$ of $N$, with $\mu$ its restriction to $G$. We also use the notation that
	$$[x]^r = \begin{cases}
	\prod\limits_{j=0}^{r-1} (x+j\hbar) & r>0\\
	1 & r=0\\
	\prod\limits_{j=1}^{|r|} (x-j \hbar) & r < 0\\
	\end{cases}.$$ It is worth noting that if $t^\lambda$ maps $p$ outside $N(\cO)$ then $E_{p,\lambda}$ will necessarily have a vector that transforms trivially under $\C^\times$, i.e. $e(E_{p,\lambda})=0$. 
	By Eq. \eqref{eq:dressedmonopole} the result follows.
\end{proof}
\begin{remark}
	The above localization computations and the ``abelianization procedure" appear in \cite{BFNSlices} as an embedding of the algebra $\tilde{\cA}^\hbar$ to an algebra of differential(-difference) operators on the maximal torus $T\subset G$.
\end{remark}

\section{Torus Links and the Spherical RCA}
\label{sec:toruslinks}
Fix $G = GL_n, N = \Ad\oplus V, G_F = \C^\times_{dil}$ and set $\cR = \cR_{G, N}, \cA^\hbar = \cA_{G,N}^\hbar$.  We focus on the case of $v\in N(\cO)$ corresponding to positive $(n,k)$ torus knots, which can be realized by the plane curve singularities $\hC_{(n,k)}$ associated to $f = x^n - t^k$. Based on the relation between Hilbert schemes of points on $\hC_{n,k}$ and GASF in Theorem \ref{thm:maincomparison}, we see that $$\cM_{(n,k)}:=\Hilb^\bullet(\hC_{(n,k)}) = M_v$$ for, \emph{e.g.}, $v = (\gamma, e_n) $ with $$\gamma=\begin{pmatrix}
0 & 1 & \cdots & 0 & 0\\
\vdots & \vdots & \ddots & \ddots & \vdots\\
0 & 0 & \ddots & 1 & 0\\
0 & 0 & \cdots & 0 & 1\\
t^k & 0 & \cdots & 0 & 0\\
\end{pmatrix}.$$

In this section we speculate on the relation with a conjecture of Oblomkov-Rasmussen-Shende \cite{ORS} concerning the relation between the homology of the Hilbert schemes of points on plane curve singularities and minimal $a$-degree (unreduced) HOMFLY-PT homology $\overline{\mathscr{H}}_{\text{min}}$ of the associated link.
\begin{conjecture}[\cite{ORS}]
	\label{conj:ORS}
	Let $\hC$ be the germ of any plane curve singularity, then there is an isomorphism of graded vector spaces $$H_*(\Hilb^\bullet(\hC)) \cong \overline{\mathscr{H}}_{\text{min}}(\text{Link of } \hC).$$
\end{conjecture}
In this isomorphism, the $q$-grading on $\overline{\mathscr{H}}_{\text{min}}$ corresponds to the number of points on $\Hilb^\bullet(\hC)$ and the $t$-grading is homological on both sides. Oblomkov-Rasmussen-Shende further conjecture that the higher $a$-degrees can be obtained from looking at the ``incidence varieties" discussed in Remark \ref{rmk:incidence}. The work of Hogancamp-Mellit shows that this conjecture (together with its generalization to higher $a$-degree) is true for torus knots \cite{HM}. 

The subsequent work of Gorsky-Oblomkov-Rasmussen-Shende \cite{GORS} reframes the Oblomkov-Rasmussen-Shende conjecture  for the $(n,k)$ torus \emph{knot} $T_{(n,k)}$ in terms of the representation theory of the rational Cherednik algebra $\cH_n$ for $\fs\fl_n$, at parameter $k/n$. The algebra $\cH_n$ is identified with the subalgebra of $\overline{\cH}_n$ generated by the differences $x_i - x_j$, $y_i-y_j$, and the permutations $\sigma \in \fS_n$, and there is an isomporphism $\overline{\cH}_n \cong \cH_n \otimes_\C \mathcal{D}$, where $\mathcal{D}$ is the algebra of differential operators on $\C$ and is generated by $X = \sum_i x_i$ and $Y = \sum_i y_i$ \cite{GORS}. When $\gcd(n,k) = 1$, the algebra $\cH_n$ has a unique (non-trivial) finite-dimensional simple module denoted $L_{k/n}$ \cite{BEG}. Since $\overline{\cH}_n \cong \cH_n \otimes_\C \mathcal{D}$, it follows that it does not have any (non-trivial) finite dimensional simple modules but does have a natural class of simple modules of the form $M \otimes_\C \C[X]$ for $M$ a simple module of $\cH_n$; in particular, there is the simple module $\overline{L}_{k/n} \cong L_{k/n} \otimes_\C \C[X].$

The algebras $\overline{\cH}_n$ and $\cH_n$ are graded with $|x_i| = - |y_i| = 1$ and $|\sigma| = 0$, which is actually an inner grading with respect to $H = \tfrac{1}{2} \sum_i (x_i y_i + y_i x_i)$. Moreover, these algebras are filtered by total polynomial degree, \emph{i.e.}, such that the $i$-th filtered component is spanned by elements that can be written by products of at most $i$ $x$'s and $y$'s. Gorsky-Oblomkov-Rasmussen-Shende conjecture that the (reduced) HOMFLY-PT homology of $T_{(n,k)}$ can be obtained from the representation $L_{k/n}$; we only state the part relevant for minimal $a$-degree.

\begin{conjecture}[\cite{GORS}]
	\label{conj:GORS}
	The simple module $L_{k/n}$ admits a filtration $\cF_\bullet$ compatible with the filtration on the rational Cherednik algebra and the grading induced by $H$ such that $$\mathscr{H}_{\text{min}}(T_{(n,k)}) \cong \Hom_{\fS_n}(\C, \gr^\cF L_{k/n}) \cong \gr^\cF e L_{k/n}$$ as graded vector spaces. 
\end{conjecture}

This conjecture can be equivalently formulated in terms of the unreduced homology $\overline{\mathscr{H}}_{\text{min}}$ after replacing $e L_{k/n}$ with $e \overline{L}_{k/n}$ \cite{GORS}.

The work of Oblomkov-Yun shows that the action of $\cH_n$ on $L_{k/n}$ can be realized geometrically via affine Springer theory \cite{OY}. They consider the equivariant (co)homology of the parabolic affine Springer fiber over an $\mathfrak{sl}_n(\cO)$ matrix with characteristic polynomial $x^n - t^k$, and identify the module $L_{k/n}$ with the associated graded with respect to the perverse filtration \cite{YunThesis} on the Springer fiber after specializing the equivariant parameter to $1$. The perverse filtration itself yields a filtration $\cF^{\text{geom}}$ on the module $L_{k/n}$ that is compatible in the sense of the above conjecture with the filtration on $\cH_n$ (we do not discuss the filtrations $\cF^{ind}$ and $\cF^{alg}$ in \cite{GORS} as they do not appear in our setup any more naturally as they do in \cite{OY}). Thus, the conjecture of Gorsky-Oblomkov-Rasmussen-Shende with $\cF = \cF^{\text{geom}}$ can be seen as a consequence of the Oblomkov-Rasmussen-Shende.

All of these ingredients appeared in the previous sections: we realized $\Hilb^\bullet(\hC_{(n,k)})$ and its parabolic generalization $\PHilb^\bullet(\hC_{(n,k)})$ as generalized affine Springer fibers in Section \ref{sec:hilb} and in Section \ref{sec:sca} we constructed actions of the spherical rational Cherednik algebra $e \overline{\cH}_n e$ and full rational Cherednik algebra $\overline{\cH}_n$ on their equivariant homologies, respectively. Moreover, the perverse filtration described above comes from the grading by number of points on $\Hilb^\bullet(\hC_{(n,k)})$ and $\PHilb^\bullet(\hC_{(n,k)})$ via an Abel-Jacobi map \cite{MS,MY}, in particular the filtration is split in the Hilbert scheme case as follows e.g. from \cite{Ren}. 

By the above discussion, it is natural to expect that the equivariant homologies of $\Hilb^\bullet(\hC_{(n,k)})$ and $\PHilb^\bullet(\hC_{(n,k)})$ realize the modules $e \overline{L}_{k/n}$ and $\overline{L}_{k/n}$, respectively; we will see below that this is exactly the case.

\begin{remark}
It is still unclear whether the rational Cherednik algebra acts naturally on the triply graded homologies of algebraic links. Of course, assuming the ORS conjecture's validity (which we have in the toric cases), one has such an action {\em par transport de structure}. It would be interesting to know what this action means in terms of knot homology. For some speculations one can consult \cite{GORS}.
\end{remark}

\begin{remark}[For the physically minded reader]
As described in the Introduction, the physical setup of \cite{DGHOR} realizes each $a$-degree of the HOMFLY-PT homology of a link $L$ (identified with the closure of an $n$-strand braid) in two different ways (exchanged by 3d mirror symmetry) as a supersymmetric Hilbert space (subject to a boundary condition $\mathcal{B}_L$ that encodes the link and a line operator $\cL_\ell$ determining the $a$-degree) in the topological $A$- or $B$-twists the 3d $\cN = 4$ $U(n)$ gauge theory with hypermultiplets in $T^*N = T^*(\Ad \oplus V)$.

In either construction, there is a natural action on the supersymmetric Hilbert spaces for the minimal $a$-degree by the algebra of local operators, identified in this case with the algebra of functions on $\Hilb^n(\C^2)$. For general links, however, the boundary condition is not compatible with an Omega-background, \emph{i.e.}, it is not compatible with the quantization of $\Hilb^n(\C^2)$ by $e \overline{\cH}_n e$. Links associated to quasi-homogeneous singularities, including the torus links $T_{(n,k)}$, are exceptions to this rule.

For higher $a$-degrees, there similarly should be an action of the algebra of local operators bound to the line operator $\cL_\ell$ but this algebra is not well understood. One could instead consider the algebra of local operators bound to the line operator $\cL_{RCA}$ used by Braverman-Etingof-Finkelberg \cite{BEF}. This algebra of local operators realizes the endomorphisms of the Procesi bundle $P \to \Hilb^n(\C^2)$ and its quantization is the full rational Cherednik algebra $\overline{\cH}_n$. For the same reason as above, there should be an action of these bundle endomorphisms on the corresponding supersymmetric Hilbert space for any link but \emph{not} necessarily its quantization. Nonetheless, it is natural to expect that the action of local operators bound to the line operator $\cL_\ell$ on the $\ell$-th $a$-degree of HOMFLY-PT homology of a link $L$ to arise as a partial symmetrization of the the action of $\cL_{RCA}$ on the supersymmetric Hilbert space in the presence of the boundary condition $\mathcal{B}_L$.
\end{remark}

\subsection{Computation of the convolution action}
The assumptions of Lemma \ref{lemma:monopoleAction} hold for $(n,k)$ torus knots due to Lemma \ref{lemma:nkstab}. The $L_v$-fixed points are labeled by cocharacters as described in Proposition \ref{prop:fixedpoints}. Moreover, the results of \cite{GKM} apply to show the Hilbert schemes are equivariantly formal (see in particular \cite[Corollary 7.25]{GSV}) so we are in a setup to apply Proposition \ref{prop:torusloc}. We will label the fixed point classes inside $\C[\hbar^{\pm1}]H^{\C^\times}_*(\cM_{(n,k)})$ by $\ket{A}.$ The map $L_v \to G_F \times \C^\times_{rot} \cong \C^\times_{dil} \times \C^\times_{rot}$ is realized by $\nu \mapsto (\nu^{-k}, \nu^n)$. This implies the relation $(n m + k \hbar) \ket{A} = 0$ for all $A$. We explicitly solve this by replacing $m \ket{A} = -\tfrac{k}{n}\hbar \ket{A}$. Let $\varphi_a$, $a = 1, ..., n$ be the components of $\varphi$ in the standard basis.
\begin{lemma}
	The action of $\C[\ft]$ is given by $$\varphi_a\ket{A} = \bigg((n - a)\tfrac{k}{n}-A_a\bigg)\hbar \ket{A}$$ and the action of $[t^\lambda]$ is given by $$[t^\lambda]\ket{A} = \bigg(\prod_{\lambda_a < 0} \prod\limits_{\alpha = 0}^{|\lambda_a|-1}((n-a)\frac{k}{n} -A_a +\alpha)\hbar\bigg)\bigg(\prod\limits_{\lambda_a > \lambda_b}\prod\limits_{\beta=0}^{\lambda_a-\lambda_b-1}((b-a+1)\frac{k}{n} -A_a+A_b +\beta)\hbar\bigg)\ket{A + \lambda}.$$
\end{lemma}
\begin{proof}
	This is a direct application of Lemma \ref{lemma:monopoleAction}.
\end{proof}
Using these ingredients and equation \eqref{eq:monopoleAction} one can obtain an expression for the action of any $[\cR_{\leq\lambda}][f]$.	
Therefore, for $\lambda$ a minuscule cocharacter we have
\begin{equation}\label{eq:SCAaction}
[\cR_{\leq \lambda}][1] \ket{A} = \sum\limits_{\lambda' \in W \cdot \lambda} \frac{\bigg(\prod\limits_{\lambda_a' < 0} \prod\limits_{\alpha = 1}^{|\lambda_a'|}(\varphi_a - \alpha \hbar)\bigg)\bigg(\prod\limits_{\lambda_a' > \lambda_b'}\prod\limits_{\beta=1}^{\lambda_a'-\lambda_b'}(\varphi_b - \varphi_a + m-\beta \hbar)\bigg)}{\bigg(\prod\limits_{\lambda_a' > \lambda_b'}\prod\limits_{\gamma=1}^{\lambda_a'-\lambda_b'}(\varphi_b - \varphi_a-\gamma \hbar)\bigg)}\ket{A + \lambda'}.
\end{equation} 
There is a similar expression for the action of $[\cR_{\leq\lambda}][f]$ for $f(\varphi, m, \hbar)$ a $W_\lambda$-invariant function:
\begin{equation}\label{eq:SCAaction2}
[\cR_{\leq \lambda}][f] \ket{A} = \sum\limits_{\lambda' \in W \cdot \lambda} \frac{(w.f)(\varphi_a, m , \hbar)\bigg(\prod\limits_{\lambda_a' < 0} \prod\limits_{\alpha = 1}^{|\lambda_a'|}(\varphi_a - \alpha \hbar)\bigg)\bigg(\prod\limits_{\lambda_a' > \lambda_b'}\prod\limits_{\beta=1}^{\lambda_a'-\lambda_b'}(\varphi_b - \varphi_a + m-\beta \hbar)\bigg)}{\bigg(\prod\limits_{\lambda_a' > \lambda_b'}\prod\limits_{\gamma=1}^{\lambda_a'-\lambda_b'}(\varphi_b - \varphi_a-\gamma \hbar)\bigg)}\ket{A + \lambda'},
\end{equation}
where $w \in W$ is such that $w.\lambda = \lambda'$; since $f$ is assumed to be invariant under $W_\lambda$, the above expression is independent of the choice of such a $w$.

\begin{proposition}
	Comparing to \cite[A(iii)]{BFNSlices}, we have an identification (up to numerical factors) $E_r[f] = [\cR_{\leq\lambda_r}][f]$ and $F_r[f] = [\cR_{\leq-\lambda_r}][\tilde f]$	 where $\lambda_r = (1,1,...,1,0,0...,0)$ with $r$ 1's and $\tilde f(\varphi) = f(\varphi-\hbar)$.
\end{proposition}
Using this presentation of the algebra, the following result is straightforward.
\begin{lemma}
	\label{lemma:module}
	For coprime $(n,k)$, $H^{\C^\times}_*(\cM_{(n,k)})$ is irreducible as a module for the spherical rational Cherednik algebra at parameter $m = -\tfrac{k}{n} \hbar.$
\end{lemma}
\begin{proof}
	We show that this module is irreducible by identifying the unique singular vector, namely $\ket{0}$. Recall that being a singular vector for the spherical rational Cherednik algebra corresponds to being in the kernel of all $F_r[f] = [\cR_{\leq-\lambda_r}][f]$. First consider the kernel of $F_n[f]$, or the classes corresponding to the cocharacter $\lambda = (-1,-1,...,-1)$. The choice of $f$ is that of a $W$ invariant polynomial $f(\varphi, m, \hbar)$. From the action given in \eqref{eq:SCAaction}, we find that
	$$F_n[1] \ket{A} = \prod \limits_{b = 1}^n \big( \varphi_b - \hbar\big) \hbar\ket{A - (1,1,..., 1)} = \prod \limits_{b = 1}^n \big((n-b)\frac{k}{n} - A_b\big) \hbar\ket{A - (1,1,..., 1)}.$$
	Since $\gcd(n,k)=1$, the factor $((n-b)\tfrac{k}{n} - A_b)$ can only vanish for $b = n$ and $A_n = 0$. It follows that the kernel of $F_n[1]$ is exactly those classes $\ket{A}$ with $A_n = 0$. The stabilizer of $\lambda = (-1,-1,...,-1)$ is all of $W = S_n$, so the operators $F_n[f]$ are labeled by choices of $W$-invariant polynomials $f(\varphi_a, m, \hbar)$. Using Eq. \eqref{eq:SCAaction2}, we find that $$F_n[f] \ket{A} = f F_n[1]\ket{A}$$ so $\ker F_n[1] \subset \ker F_n[f]$ for all $f$.	
	
	Now consider the action of $F_{n-1}[f]$ on sums of fixed point classes with $A_n = 0$. Using Eq. \eqref{eq:monopoleAction} for we have, after a dramatic simplification following from $A_n = 0$,
	$$F_{n-1}[1] \ket{A_1,...,A_{n-1}, 0} = \bigg(\prod\limits_{b=1}^{n-1}\big((n-1-b)\frac{k}{n} - A_b\big)\hbar\bigg) \ket{A_1 - 1,..., A_{n-2}-1,0}.$$
	Again, since $\gcd(n,k)=1$, the factor $((n-1-b)\frac{k}{n} - A_b)$ can only vanish for $b = n-1$ and $A_{n-1} = 0$. Therefore $\ket{A_1, ..., A_{n-1},0}$ is in the kernel of $F_{n-1}[1]$ if and only if $A_{n-1} = 0$. Thus $\ker F_n[1] \cap \ker F_{n-1}[1]$ only contains classes with $A_n = A_{n-1} = 0$. 
	
	The action of $F_{n-1}[f]$ on $\ket{A_1, ..., A_{n-1},0}$ enjoys an equally dramatic simplification and we find that $(\ker F_{n}[1] \cap \ker F_{n-1}[1]) \subset \ker F_{n-1}[f]$ for any choice of $f$. Continuing this process shows that $$\ker F_n[1] \cap \ker F_{n-1}[1] \cap ... \cap F_{1}[1] = {\rm span}\{\ket{0}\}$$ and that $\ket{0}$ belongs to the kernel of all $F_r[f]$.
\end{proof}
Now we state and prove the main theorem of this section.
\begin{theorem}
	\label{thm:glnirrep}
	For coprime $(n,k)$, $H_*^{\C^\times}(\cM_{(n,k)})$ can be identified with the irreducible 
	representation $e \overline{L}_{k/n}$ of the spherical rational Cherednik algebra of $\fg\fl_n$ at 
	parameter $m = -\tfrac{k}{n} \hbar.$ That is, setting the equivariant parameter $\hbar$ in $H_*^{\C^\times}(\cM_{(n,k)})$ to $-1$, the quotient algebra
	$e\cH_n e/(m-\tfrac{k}{n})$ acts on $H_*^{\C^\times}(\cM_{(n,k)})$.
\end{theorem}
\begin{proof}
	From \cite{KN}, or a direct computation using \eqref{eq:SCAaction}, it follows that for all $n$ the operators $X = [\cR_{\leq (1,0,...,0)}] = E_1[1]$ and $Y = [\cR_{\leq (-1,0,...,0)}] = F_1[1]$ generate an appropriately scaled copy of the Heisenberg algebra: $[X, Y] = n \hbar.$ Since we have shown that $H_*^{\C^\times}(\cM_{(n,k)})$ is irreducible as a module for the spherical rational Cherednik algebra of $\fg\fl_n$ at parameter $m = -\tfrac{k}{n} \hbar$ it follows that it must decompose as a product $\C[X] \otimes M$, where $M$ is some irreducible module for the spherical rational Cherednik algebra of $\fs\fl_n$ \cite{GORS}. Finally, noting that the spherical rational Cherednik algebra of $\fs\fl_n$ at parameter $m = -\tfrac{k}{n} \hbar$ has a unique finite dimensional, irreducible module, it suffices to show that $\ker Y \simeq M$ is finite dimensional.
	
	The equivariant homology $H^{\C^\times}_*(\cM_{(n,k)})$ above is infinite dimensional. Nonetheless, it has finite dimensional graded components, where we grade by the number of points $d$ on the Hilbert scheme, or, equivalently, the degree in affine Grassmannian $\Gr_{G}$. Consider the corresponding graded Euler character:
	$$\chi_q(\cM_{(n,k)}) := \sum_{d \geq 0} q^d \chi(\cM^d_{(n,k)}),$$
	which can easily be computed from counting fixed points.
	
	Recall that the fixed points in $\cM_{(n,k)}$ are labeled by cocharacters $A$ as in Prop. \ref{prop:fixedpoints}, denote the set of such $A$ by $\mathfrak{A}_{(n,k)}$. The fixed point labeled by $A$ belongs to the component of the Hilbert scheme with $$d(A) = \sum \limits_{a = 1}^n A_a$$ and one finds
	$$\chi_q(\cM_{(n,k)}) = \sum\limits_{A \in \mathfrak{A}_{(n,k)}} q^{d(A)}.$$
	From Prop. \ref{prop:fixedpoints}, we have $$\mathfrak{A}_{(n,k)} = \bigsqcup_{i \geq 0} \mathfrak{A}^i_{(n,k)},$$
	where $$\mathfrak{A}^i_{(n,k)} = \{A = (A_1', ..., A_{n-1}',0)+ i(1,1,...,1) | 0 \leq A_{n-1}' \leq ... \leq A_1' \leq k\}.$$
	In particular, 
	$$\chi_q(\cM_{(n,k)}) = \frac{1}{1-q^n} \sum \limits_{A \in \mathfrak{A}^0_{(n,k)}} q^{d(A)}$$
	
	We can identify an element $A = (A_1', ..., A_{n-1}',0) \in \mathfrak{A}^0_{(n,k)}$ with a partition of $d(A)$ that fits into a rectangle of size $(n-1) \times k$. The partitions are exactly counted by the $q$-binomial coefficient
	$$\sum \limits_{A \in \mathfrak{A}^0_{(n,k)}} q^{d(A)} = \left[\begin{array}{c} n + k - 1\\ n-1\\ \end{array}\right]_q := \prod \limits_{j=0}^{k} \frac{1-q^{n-1+j-j}}{1-q^{j+1}},$$
	from which we conclude
	$$\chi_q(\cM_{(n,k)}) = \frac{1}{1-q^n}\left[\begin{array}{c} n + k - 1\\ n-1\\ \end{array}\right]_q.$$
	
	Noting that $X$ changes $q$-degree by 1, we can determine the dimension of $M$ by multiplying the above by $1-q$, the reciprocal of the graded dimension of $\C[X]$, and setting $q = 1$. One finds $$\dim_\C M = \frac{1}{n}\binom{n+k-1}{n-1} = \dim_\C H^*({\overline \cJ_{n,k}}),$$ where ${\overline \cJ_{n,k}}$ is the compactified Jacobian of the curve $\hC_{n,k}$. Indeed, this is the dimension of $L_{k/n}$ \cite{GORS}.
\end{proof}

\begin{remark}
	It is worth noting that $\tilde{\cA}^\hbar$ is bigraded by the degree in $\Gr_G$, called ``monopole number" in the physics literature, and by the action induced by scaling $\C[\ft,\hbar]^W$ with weight 2, called ``R-charge" in the physics literature. In particular, we assign the degree $(\pm r, r+2\deg f)$ to $[\cR_{\leq\pm\lambda_r}][f]$. The spherical rational Cherednik algebra of $\fg\fl_n$ is also bigraded by the difference in degree of $x$'s and $y$'s and the total polynomial degree. That the respective filtrations agree follows from \cite{KN}. Note that after specializing $\hbar$ the latter of these gradings becomes merely a filtration.
\end{remark}

\subsection{Comparison to results of Gorsky-Simental-Vazirani}
\label{sec:gsv}
In the recent preprint \cite{GSV}, when $\hC=\{x^n=t^k\},\, \gcd(n,k)=1$, another action of the rational Cherednik algebra of $\fg\fl_n$ (see Definition \ref{def:cherednikalg})
is defined on the localized equivariant (Borel-Moore) homology of the parabolic flag Hilbert schemes $\PHilb_{m,m+n}(\hC)$. (In the above notation, this would correspond to $\PHilb^{m,m+(1,1,...,1)}(\hC)$.)
By Definition \ref{def:parabolic} and Theorem \ref{thm:maincomparison}, $$\PHilb^\bullet(\hC):=\bigsqcup_{m\geq 0}\PHilb_{m,m+n}(\hC) \cong \tM_v := M_v^{\bI,\Lie(\bI)\oplus \cO^n}$$ where $\bI$ is the standard Iwahori of $G_\cK$ and $v$ is associated to $\hC$ as in Theorem \ref{thm:maincomparison}.

The action in {\em loc. cit.} is defined using the fixed point basis and stems from combinatorial and representation-theoretic considerations. It was motivated by both the results of \cite{WebDO} as well as this work, which was in preparation at the time. In what follows, we show that the actions defined in Theorem \ref{thm:convolution} and \cite[Theorem 7.14]{GSV} coincide. We set $\bN_\bI = \Lie(\bI) \oplus \cO^n$.
\begin{theorem}
	\label{thm:GSVcompare}
	The action in \cite[Theorem 7.14]{GSV} on the module $$H^{\C^\times}_*(\bigsqcup_{m\geq 0}\text{PHilb}_{m, m+n}(\hC))[\hbar^{-1}]$$ agrees with the action defined by Theorem \ref{thm:convolution} on 
	$H^{\C^\times}_*(\tM_v)[\hbar^{-1}]$.
\end{theorem}
\begin{proof}
	After inverting $\hbar$, the GKM localization formula implies the fixed point classes are a {\em basis} for the equivariant BM homology. 
	As proven in \cite{WebDO, WebLP, GSV}, the rational Cherednik algebra $\cH_n$ is generated by the Dunkl-Opdam subalgebra, the finite symmetric group $\Sn$, as well as two elements $\tau, \lambda$, which can be identified with
	$\pi,\pi^{-1} \in \Sn^{aff}$ under Suzuki's embedding of $\cH_n$ to the trigonometric Cherednik algebra (see \cite{KN,GSV}). We only need to identify these generators on both sides - the relations they satisfy are proved in \cite{WebDO,WebLP,GSV}.
	
	The Springer action is induced by the following diagram 
	\begin{equation}
	\label{eq:SpringerCoulombDiagram}
	\begin{aligned}
\includegraphics{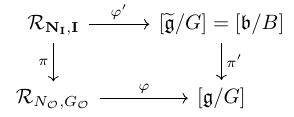}
	\end{aligned}
	\end{equation} and the action of the simple reflections $s_1,\ldots, s_{n-1}$ comes from convolution with $[\cR^{\leq s_i}_{\bN_\bI, \bI}]$, which come about via pullback from classical correspondences on the Steinberg variety. The Springer action of \cite{GSV} is the usual one coming from projections to smaller affine flag varieties which is also defined via pullback from the right. The coincidence of the two is a classical result.
	
	The equivariant cohomology classes $u_i\in H^*_{T}(pt)$ are identified with cap product by the Chern classes $c(\cL_i)$ of the natural line bundles on the affine flag variety. The identification of these two is e.g. \cite[Lemma 5.1.6]{OY}.

	Finally, we need the $\pi=\tau$ and $\pi^{-1}=\lambda$ operators. 
	In \cite[Theorem 5.2]{WebLP} and \cite[Lemma 4.2]{WebDO}, $\tau$ is identified with convolution by the correspondence in Section \ref{sec:corr} corresponding to the space
	$$X_\tau:=\{(V_\bullet, V_\bullet')|V_i=V_{i+1}'\}=\{(g\bI,g'\bI)|g\textbf{t}=g'\}$$ and similarly $\sigma$ is identified with convolution by the correspondence 
	$$X_\sigma:=\{(V_\bullet,V_\bullet')|V_i=V_{i-1}'\}=\{(g\bI,g'\bI)|g'\textbf{t}=g\},$$ where $\textbf{t}$ is the matrix sending  $ e_i\mapsto e_{i+1}, i=1,\ldots, n-1$ and $e_n\mapsto t e_1$ in the standard basis of $\cK^n$. It is immediate that these coincide with the maps $T, \Lambda$ on $H^{\C^\times}_*(\bigsqcup_{m\geq 0}\text{PHilb}_{m, m+n}(\hC))[\hbar^{-1}]$ in \cite[Theorem 7.14]{GSV}.
\end{proof}

\subsection{Example: $(2,2\ell+1)$ torus knots}
\label{sec:SCA2}
We end with the explicit example for the case of $T_{(2,2\ell+1)}$ torus knots. We discuss the module structure of $H^{L_v}_*(\cM_{(2,2\ell+1)})$ and $H^{L_v}_*(\tilde{\cM}_{(2,2\ell+1)}).$

As described in Definition \ref{def:cherednikalg}, the rational Cherednik algebra of $\fg\fl_2$ is the quotient algebra
$$\cH_n=\frac{\C[\hbar,m]\langle x_1,x_2,y_1,y_2\rangle\rtimes\C\fS_2}{\sim}$$
where $\sim$ consists of the relations $[x_i,x_j]=[y_i,y_j]=0$ for all $i, j$, and 
$$[y_i,x_j]=\begin{cases}
-\hbar+m (12) & \text{ if } i=j,\\
-m (12)& \text{ if } i \neq j.
\end{cases}$$
The symmetrizing element is given by $e = \tfrac{1}{2}(1 + (12))$.

The spherical subalgebra has generators given by an $\mathfrak{sl}_2$ triple $$E = -\tfrac{1}{2}e (x_1^2 + x_2^2) e \qquad F = \tfrac{1}{2} e (y_1^2 + y_2^2) e \qquad H = \tfrac{1}{2}e (x_1 y_1 +y_1 x_1 + x_2 y_2 + y_2 x_2) e$$ and a Weyl pair $$X = e (x_1 +x_2)e \qquad Y = e (y_1 + y_2) e$$ transforming in the defining representation of that $\mathfrak{sl}_2$. In particular, the non-zero commutation relations between these generators are those defining $\mathfrak{sl}_2$ and the Weyl algebra $$[E,F] = \hbar H \hspace{1 cm} [H,E] = 2 \hbar E \hspace{1cm} [H,F] = - 2 \hbar F \hspace{1cm} [X,Y] = 2\hbar,$$ and those describing the way $X,Y$ transform under $\mathfrak{sl}_2$ $$[E,X] = [F,Y] = 0 \hspace{1cm} [H,X] = [E,Y] = \hbar X \hspace{1cm} [H,Y] = -[F,X] = -\hbar Y.$$
Denote $W^+ = \tfrac{1}{2}X^2, W^0=-\tfrac{1}{2}(X Y + Y X), W^-=-\tfrac{1}{2} Y^2$, so that the $W^{\pm}, W^0$ transform in the adjoint representation of the above $\mathfrak{sl}_2$. There is one additional relation amongst these operators: $$C_2 = 2(E W^- + F W^+) + H W^0 +m(m-\hbar),$$ where $C_2 = 2(E F + F E) + H^2$ is the quadratic Casimir of the $\mathfrak{sl}_2$ triple and $m$ is a complex parameter.

\begin{theorem}
	The spherical subalgebra, realized as the quantized BFN algebra $\tilde{\cA}_{G,N}^\hbar$ for $G = GL_2$, $N = \Ad\oplus \C^2$, acts via convolution on $H^{\C^\times}_*(\cM_{(2,2\ell+1)})$ for $m = -\tfrac{2\ell+1}{2}\hbar$.  As a module for the spherical rational Cherednik algebra of $\fg\fl_2$, we have $$H^{\C^\times}_*(\cM_{(2,2\ell+1)}) \simeq  e \overline{L}_{(2\ell+1)/2},$$ where $e$ is the $\fS_2$ symmetrizer in rational Cherednik algebra of $\fg\fl_2$ and $\overline{L}_{(2\ell+1)/2}$ is the simple rational Cherednik algebra (at parameter $m = -\tfrac{2\ell+1}{2} \hbar$) module induced from the trivial representation of $\fS_2$.
\end{theorem}
To simplify the expressions below, we will simply write $k$ instead of $2 \ell + 1$. The below does \emph{not} apply when $k$ is even.
\begin{proof}
	First consider the monopole operator $X:= [\cR_{(1,0)}]$. This arises from the orbit $\Gr^{(1,0)}_{GL_2}$, which form a copy of $\P^1$ parameterized by two affine charts given by $$\begin{pmatrix} t&0\\ a_1&1 \end{pmatrix} \hspace{1cm} \begin{pmatrix} 1&a_2\\ 0& t\end{pmatrix}$$
	with transition function $a_2 = \tfrac{1}{a_1}$. There are $G(\cO)$ torus fixed points at the origins of these affine charts, and the coordinate $a_1$ (resp. $a_2$) transforms with weight $\varphi_2-\varphi_1$ (resp. $\varphi_1 - \varphi_2$). Applying Eq. \eqref{eq:SCAaction} yields $$X\ket{A_1, A_2} = \frac{A_1-A_2 - k}{A_1-A_2 - \tfrac{k}{2}}\ket{A_1+1,A_2} + \frac{A_1-A_2}{A_1-A_2 - \tfrac{k}{2}}\ket{A_1,A_2+1}.$$
	
	Similarly, there is the monopole operator $Y:= [\cR_{(0,-1)}]$ coming from the orbit $\Gr^{(0,-1)}_{GL_2}$, which forms a copy of $\P^1$ parameterized by two affine charts $$\begin{pmatrix} 1&0\\ a_1 t^{-1} &t^{-1} \end{pmatrix} \hspace{1cm} \begin{pmatrix} t^{-1}&a_2 t^{-1}\\ 0& 1\end{pmatrix}$$ with transition function $a_2 = \tfrac{1}{a_1}$. The coordinate $a_1$ again transforms with weight $\varphi_1 - \varphi_2$. We find that
	$$Y \ket{A_1, A_2} = \frac{(A_1-A_2)(\tfrac{k}{2} - A_1) \hbar}{A_1-A_2 - \tfrac{k}{2}}\ket{A_1-1,A_2}+\frac{A_2(k-A_1+A_2)\hbar}{A_1-A_2-\tfrac{k}{2}}\ket{A_1, A_2-1}.$$
	
	There are two other monopole operators we will be interested in, namely $E= [\cR_{(1,1)}]$ and $F= -[\cR_{(-1,-1)}]$. They come from $\Gr_{GL_2}^{(1,1)}$ and $\Gr_{GL_2}^{(-1,-1)}$ respectively, both of which are single points. Applying Eq. \eqref{eq:SCAaction} gives $$\begin{array}{c c}
	E \ket{A_1, A_2} = \ket{A_1+1, A_2+1} & F \ket{A_1, A_2} = (\tfrac{k}{2}-A_1)A_2\hbar^2 \ket{A_1-1, A_2-1}
	\end{array}$$
	from which it is straightforward to compute that $H = \hbar - \varphi_1 - \varphi_2$ acts as
	$$H\ket{A_1, A_2} = (A_1+A_2+1-\tfrac{k}{2})\hbar \ket{A_1, A_2}$$ and makes $(E,F,H)$ an $\mathfrak{sl}_2$ triple. The quadratic Casimir $C_2 = 2(E F+F E)+H^2$ acts as $$C_2 \ket{A_1, A_2} = \bigg((A_1-A_2 - \tfrac{k}{2})^2-1\bigg)\hbar^2 \ket{A_1, A_2}.$$ It is straightforward to check that the desired relations are indeed satisfied with $m = -\tfrac{k}{2}\hbar$.
	
	From the action of $\mathfrak{sl}_2$, we see that the classes $\ket{A_1, 0}$ are lowest weight vectors with weights $\nu = (A_1+1-\tfrac{k}{2})\hbar$. Therefore, the homology of this GASF can be expressed as an $\mathfrak{sl}_2$ module as $$H_*^{\C^\times} (\cM_{(2,k)}) = \bigoplus_{A_1 = 0}^{k} \Lambda_{\big(A_1+1-\tfrac{k}{2}\big)\hbar},$$ where $\Lambda_{\nu}$ is the $\mathfrak{sl}_2$ Verma module generated by a lowest weight vector of weight $\nu$. It is also worth noting that $\ket{0, 0}$ is a vacuum vector for the Heisenberg algebra generated by $X,Y$; hence it is the unique spherical rational Cherednik algebra singular vector. We can therefore identify this with the SCA module:$$H_*^{\C^\times} (\cM_{(2,k)}) \simeq e \overline{L}_{k/2},$$ where $e$ is the $\fS_2$ symmetrizer in the rational Cherednik algebra and $\overline{L}_{k/2}$ is the simple rational Cherednik algebra module induced from the trivial representation of $\fS_2$.
\end{proof}

\begin{remark}
	It is worth noting that there is another presentation of the spherical rational Cherednik algebra for $\fg\fl_2$ given by a (different) $\fsl_2$-triple $(\tilde E, \tilde F, \tilde H)$ and the Weyl pair $X,Y$.\footnote{The change of variables is given by $\tilde E = E - \tfrac{1}{4} X^2, \tilde F = F + \tfrac{1}{4} Y^2, \tilde H = H + \tfrac{1}{4}(X Y + Y X)$.} In this presentation, $X,Y$ transform trivially under $\fsl_2$ and the quadratic Casimir of the $\fsl_2$-triple is given by $$
	\tilde C_2 = (m-\tfrac{3}{2}\varepsilon)(m+\tfrac{1}{2}\varepsilon)$$ with no other constraints. In this presentation we find that the homology of our generalized affine Springer fiber is given by $$H_*^{\C^\times} (\cM_{(2,2\ell+1)}) \simeq \C[X] \otimes \Sym^\ell \square,$$ where $\Sym^\ell \square$ is the $\ell+1$ dimensional representation of $\fsl_2$. We can identify $\Sym^\ell \square$ as the cohomology of $\P^{\ell}$, the compactified Jacobian for the $(2,2\ell+1)$ torus knots. This feature was predicted in \cite{ORS}. 
\end{remark}

We now move to the action of the rational Cherednik algebra on the homology of parabolic Hilbert schemes. In particular, we spell out the comparison in Theorem \ref{thm:GSVcompare} between the action given by Theorem \ref{thm:convolution} and \cite{GSV}.
\begin{theorem}
	The action of the rational Cherednik algebra on the homology of $\PHilb^\bullet(\hC)$ given in Theorem \ref{thm:convolution} agrees with the action of \cite[Theorem 7.14]{GSV}.
\end{theorem}
\begin{proof}
	As discussed at the end of Section \ref{sec:hilb}, we describe the action on classes $\ket{A,\sigma}$ associated to the fixed points $\sigma t^A p$ and match the action of the rational Cherednik algebra given in \cite{GSV} by identifying these fixed points with their ``renormalized basis." We start by identifying
	$$\ket{A_1, A_2,()} = \tilde{v}_{(A_1,A_2)} \qquad \ket{A_1,A_2,(12)} = \tilde{v}_{(A_2,A_1)}.$$
	
	The action of the equivariant parameters $\varphi_a$ on the class $\ket{A,\sigma}$ can be easily seen to be
	$$\varphi_1 \ket{A,()} = (\tfrac{k}{2}-A_1) \hbar \ket{A,()} \qquad \varphi_1 \ket{A,(12)} = -A_2\hbar \ket{A,(12)}$$
	and
	$$\varphi_2 \ket{A,()} = -A_2 \hbar \ket{A,()} \qquad \varphi_2 \ket{A,(12)} = (\tfrac{k}{2} - A_1) \hbar \ket{A,(12)},$$
	which translates to (for $A_2 \leq A_1$)
	$$\varphi_1 \tilde{v}_{(A_2, A_1)} = (\tfrac{k}{2}-A_1) \hbar \tilde{v}_{(A_1, A_2)} \qquad \varphi_2 \tilde{v}_{(A_1, A_2)} = - A_2 \hbar \tilde{v}_{(A_2, A_1)}$$
	and (for $A_1 > A_2$)
	$$\varphi_1 v_{(A_2, A_1)} = -A_2 \hbar v_{(A_2, A_1)} \qquad \varphi_2 v_{(A_2, A_1)} = (\tfrac{k}{2}-A_1) \hbar v_{(A_2, A_1)}$$
	we can thus identify $u_1 = \varphi_1$ and $u_2 = \varphi_2$ in \cite[Theorem 7.14]{GSV}.
	
	The action of the transposition $s$ on $\ket{A,\sigma}$ is given by
	$$
	\begin{aligned}
	s \ket{A_1,A_2,()} = \frac{k}{2(A_2-A_1)-k}\ket{A_1,A_2,()}+\frac{2(A_2-A_1)}{2(A_2-A_1)-k}\ket{A_1,A_2,(12)}\\
	s \ket{A_1,A_2,(12)} = \frac{2(A_2-A_1+k)}{2(A_2-A_1)-k}\ket{A_1,A_2,()}-\frac{k}{2(A_2-A_1)-k}\ket{A_1,A_2,(12)}\\
	\end{aligned}.
	$$
	from which it follows that $1-s$ acts as
	$$
	\begin{aligned}
	(1-s) \ket{A_1, A_2,()} = \frac{2(A_1-A_2)}{2(A_2-A_1)-k}(\ket{A_1,A_2,()}-\ket{A_1,A_2,(12)})\\
	(1-s) \ket{A_1,A_2,(12)} = \frac{2(A_1-A_2-k)}{2(A_2-A_1)-k}(\ket{A_1,A_2,(12)}-\ket{A_1,A_2,()})\\
	\end{aligned}
	$$
	In agreement with the action of $1-s$ in \cite[Theorem 7.14]{GSV}.
	
	Finally, the actions of $T$ and $\Lambda$ do not require a fancy localization formula as they correspond to point classes in the affine flag variety. In particular, we find that the excess intersection factors are trivial for $T$:
	$$
	T \ket{A_1,A_2,()} = \ket{A_1+1,A_2,(12)} \qquad T \ket{A_1,A_2,(12)} = \ket{A_1,A_2+1,()}
	$$
	and they are $-A_2\hbar$ (resp. $(\tfrac{k}{2}-A_1) \hbar$) for $\Lambda$ on $\ket{A_1,A_2,()}$ (resp. $\ket{A_1,A_2,(12)}$):
	$$
	\Lambda \ket{A_1,A_2,()} = (\tfrac{k}{2}-A_2)\hbar \ket{A_1,A_2-1,(12)} \qquad \Lambda \ket{A_1,A_2,(12)} = (\tfrac{k}{2}-A_1)\hbar \ket{A_1,A_2-1,()}
	$$
	in agreement with the action of $T,\Lambda$ on $\tilde{v}_{\textbf{a}}$ from \cite[Theorem 7.14]{GSV}.
\end{proof}

\subsection*{Acknowledgements}
The authors thank Tudor Dimofte and Eugene Gorsky for discussions that initiated this project as well as for comments and for urging us to publish our results. We also thank Justin Hilburn, Joel Kamnitzer, and Alex Weekes for sharing their preliminary results in \cite{HKW}, and Jos\'e Simental Rodriguez and Minh-Tam Trinh for comments on a draft of this paper. N.G. would like to thank Ingmar Saberi and Jos\'e Simental Rodriguez for useful conversations.

Part of this work was carried out during the KITP program Quantum Knot Invariants and Supersymmetric Gauge Theories (fall 2018), supported by NSF Grant PHY-1748958.

\newpage
\appendix
\section{}
\label{appendix}
\subsection{Restriction with supports}
\label{sec:rws}
In this section, we define the restriction with support homomorphisms used in the definition of $p^*$ in Theorem \ref{thm:convolution}. We follow \cite{BFN}. 
\begin{definition}
	\label{def:rws}
	Suppose we have a Cartesian diagram of ind-varieties
	\begin{center}
\includegraphics{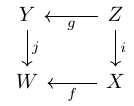}
	\end{center}
	and let $A,B$ be (possibly unbounded) complexes of constructible sheaves on $W,X$. Then 
	suppose we are given $\varphi\in \Hom(A,f_*B)\cong \Hom(f^*A,B)$. Define the morphism of complexes 
	$$j^! A\to j^!f_*f^* A\cong g_*i^!f^* A\to g_*i^! B$$ as the composition of the adjunction map and $\varphi$.
	This induces a map on hypercohomology: $$H^*(Y,j^!A)\to H^*(Z,i^!B).$$ We will call this map ``restriction with supports". 
\end{definition}

\begin{remark}
	Suppose we have a Cartesian diagram of varieties
	\begin{center}
\includegraphics{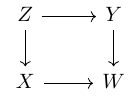}
	\end{center}
	If the first arrow is a regular embedding, let $N$ be the pullback to $Z$ of the normal bundle $N_{X/W}$. There is a specialization map 
	$$\sigma: H_*(Y)\to H_*(N), \; [V]\mapsto [C_{(C\cap Z)/V}].$$ The usual refined intersection map/pullback with support is defined  as the composition $H_*(Y)\to H_*(N)\to H_*(Z)$.
\end{remark}

\subsection{Finite-dimensional approximation}
In many parts of this paper, we consider equivariant complexes on infinite-dimensional ind-varieties, in particular $\cR, \cT$ and $N_\cO$ and their substacks. We refer the reader to \cite[Section 2]{BFN} for more precise definitions in the first two cases, and in the latter case define $$D^b_{\tG_\cO}(N_\cO)$$ to be the direct limit over the finite-dimensional approximations to $N_\cO$ given by $N_\cO/t^iN_\cO$. The degree shifts such as $[-2\dim N_\cO]$ we use, are also to be understood as in \cite[Section 2]{BFN}.
\subsection{Associativity}
In this section, we prove that the convolution product defined in Theorem \ref{thm:convolution} is associative. We follow the proof of associativity of the convolution product of $\tilde{\cA}^\hbar_{G,N, \bP, \bN}:= \tilde{\cA}^\hbar_{\bP, \bN}$ in \cite[Section 3]{BFN} and the rough outline in the preprint \cite{HKW} .

\begin{lemma}
	\label{lemma:associativity}
	The convolution product defined in Theorem \ref{thm:convolution} is associative.
\end{lemma}
\begin{proof}
	We consider the following commutative diagram, which is a
	`product' of the upper row of \eqref{eq:convolutiondiagram} and the appropriate version of \cite[(3.2)]{BFN}:
\begin{equation}
	\label{eq:associativitydiagram}
	\begin{aligned}
\includegraphics{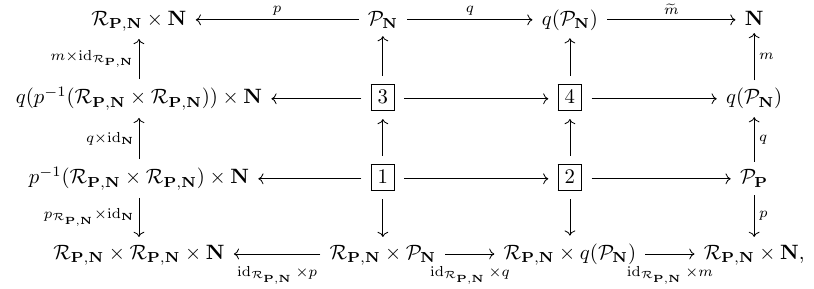}
	\end{aligned}
	\end{equation}
	where we have defined
	\begin{equation*}
	\boxed{1} = \{ (g_1,g_2,v')\in \tG_{\cK}^{\cO}\rtimes \C^\times_{rot}\times \tG_{\cK}^{\cO}\rtimes \C^\times_{rot}\times  \bN \mid
	g_2v', g_1g_2v' \in \bN \},
	\end{equation*}
	and $\boxed{2}$, $\boxed{3}$, $\boxed{4}$ are quotients of
	$\boxed{1}$ by $1\times \tbP\rtimes \C^\times_{rot}$, $\tbP \rtimes \C^\times_{rot}\times 1$, $\tbP \rtimes \C^\times_{rot}\times
	\tbP\rtimes \C^\times_{rot}$ respectively. Here $\tbP \rtimes \C^\times_{rot}\times \tbP \rtimes \C^\times_{rot}$ acts on $\boxed{1}$ by
	\begin{equation*}
	(h_1,h_2)\cdot (g_1,g_2, v') = (g_1 h_1^{-1}, h_1 g_2 h_2^{-1},
	h_2v')\quad\text{for $(h_1,h_2)\in \tbP\times \tbP$}.
	\end{equation*}
	The horizontal and vertical arrows from $\boxed{1}$, $\boxed{4}$ are
	given by
	\begin{equation}\label{eq:14}
	\xymatrix@C=10pt{(g_1, [g_2, v'], v') &
		(g_1,g_2, v')\in\boxed{1} \ar@{|->}[l]^{p_1}
		\ar@{|->}[d]^{p_2} \\
		& ([g_1,g_1g_2v'],g_2,v'),
	} \quad
	\xymatrix@C=10pt{
		[g_1g_2, v'] &
		\\
		\boxed{4}\ni [g_1, [g_2,v']] \ar@{|->}[u] \ar@{|->}[r] &
		[g_1, g_2v'].
	}
	\end{equation}
	Arrows from $\boxed{2}$, $\boxed{3}$ are given by the obvious
	modification of above ones, as $\boxed{1}\to\boxed{3}$, etc.\ are
	fiber bundles. Also, $p_{\cR_{\bP, \bN}}$ is as defined in \cite{BFN}, i.e. 
	$$(g_1,[g_2,s])\mapsto ([g_1,g_2s],[g_2,s]).$$
	
	Let $\alpha\in H^{L_v}(M^{\bP, \bN}_v)$ and $c_1, c_2\in \tilde{\cA}_{\bP, \bN}^{\hbar}$. The convolution product $c_2\star \alpha$ is given by
	applying the construction in Theorem \ref{thm:convolution} (i.e. induced homomorphisms in BM homology) to the bottom row from left to
	right, and $c_1\star (c_2\star \alpha)$ is then  obtained by going up in the rightmost column. Similarly
	$(c_1\star c_2)\star \alpha$ is given by going up the leftmost column using the construction in \cite{BFN} and
	then from left to right along the top row.
	
	Therefore the associativity of the convolution product is the statement that the induced morphisms $$-\star(-\star -), (-\star -)\star -: \tilde{\cA}_{\bP, \bN}^{\hbar}\otimes \tilde{\cA}_{\bP, \bN}^{\hbar} \otimes H_*^{L_v}(M_v^{\bP, \bN})\to H_*^{L_v}(M_v^{\bP, \bN})$$ are equal. 
	This would follow commutativity of the associated ``large square" in BM homology. (It might be helpful for the reader to recall the usual diagram for associativity of an algebra action).

	We will in fact prove that each square is commutative after applying BM homology.

	
	Let us first look at the bottom left square. We can extend the
	square to a cube as
	\begin{equation*}
	\xymatrix@!C=100pt{
		& \tG_{\cK}^{\cO}\rtimes \C^\times_{rot}\times\cR_{\bP, \bN}\times \bN \ar@{->}'[d]^{p'\times\id_{ \bN}}[dd]
		& & \tG_{\cK}^{\cO}\rtimes \C^\times_{rot}\times \cP_{\bN} \ar@{->}[dd]^P\ar@{->}[ll]_{\id_{\tG_{\cK}^{\cO}\rtimes \C^\times_{rot}}\times p}
		\\
		p^{-1}(\cR_{\bP, \bN}\times\cR_{\bP, \bN})\times \bN \ar@{->}[ur]\ar@{->}[dd]
		& & \boxed{1} \ar@{->}[ur]\ar@{->}[ll]\ar@{->}[dd]
		\\
		& \cT_{\bP, \bN}\times\cR_{\bP, \bN}\times \bN
		& & \cT_{\bP, \bN} \times \cP_{\bN} \ar@{->}'[l]_{\id_{\cT_{\bP, \bN}}\times p}[ll]
		\\
		\cR_{\bP, \bN}\times\cR_{\bP, \bN}\times \bN\ar@{->}[ur]
		& & \cR_{\bP, \bN}\times \cP_{\bN} \ar@{->}[ll]\ar@{->}[ur]
	}
	\end{equation*}
	Arrows from spaces in the front square to those in the rear square are
	closed embeddings. Arrows in the rear square are as indicated, where we have defined $P\colon \tG_{\cK}^{\cO}\rtimes \C^\times_{rot}\times \cP_{\bN}\to \cT_{\bP, \bN} \times
	\cP_{\bN}$ by $(g_1,g_2,v')\mapsto ([g_1,g_1g_2v'],g_2,v')$, just as the downward arrow from $\boxed{1}$ above.
	
	The top, right, left and bottom faces of the cube are Cartesian and we have the isomorphisms
	\begin{equation*}
	P^*(\omega_{\cT_{\bP, \bN}}\boxtimes \pi_1^!\cF^v_{\bP, \bN})\cong
	\omega_{\tG_\cK^{\cO}\rtimes \C^\times_{rot}}\boxtimes \pi_1^!\cF^v_{\bP, \bN} 
	\end{equation*}
	\begin{equation*}
	(p'\times \id_{\bN})^*\omega_{\cT_{\bP, \bN}}\boxtimes \omega_{\cR_{\bP, \bN}}\boxtimes \cF^v_{\bP, \bN}\cong
	\omega_{\tG_\cK^\cO\rtimes \C^\times_{rot}}\boxtimes \omega_{\cR_{\bP, \bN}}\boxtimes \cF^v_{\bP, \bN}.
	\end{equation*}
	This gives us two pullbacks with supports 
	$$H^*_{\tbP \rtimes \C^\times_{rot} \times \tbP \rtimes \C^\times_{rot}}(\cR_{\bP, \bN}\times \cP_{\bN},\omega_{\cR_{\bP, \bN}}\boxtimes \pi_1^!\cF^v_{\bP, \bN})\to
	H^*_{\tbP\rtimes \C^\times_{rot}\times \tbP\rtimes \C^\times_{rot}}(\boxed{1},\omega_{\tG_\cK^\cO\rtimes \C^\times_{rot}}\boxtimes \pi_1^!\cF^v_{\bP, \bN})$$ 
	and
	$$H^*_{\tbP \rtimes \C^\times_{rot} \times \tbP\rtimes \C^\times_{rot}}(p_{\cR_{\bP, \bN}}^{-1}(\cR_{\bP, \bN}\times \cR_{\bP, \bN})\times \bN,\omega_{\cR_{\bP, \bN}}\boxtimes \omega_{\cR_{\bP, \bN}}\boxtimes \cF^v_{\bP, \bN})\to$$
	$$H^*_{\tbP \rtimes \C^\times_{rot}\times \tbP\rtimes \C^\times_{rot}}(\boxed{1},\omega_{\tG_\cK^\cO}\boxtimes \pi_1^!\cF^v_{\bP, \bN}).$$

	We claim that these are the same homomorphism. Consider
	$\omega_{\cT_\bP}\boxtimes\omega_{\cR_\bP}\boxtimes \cF^v_{\bP, \bN}$ on $\cT_{\bP, \bN}\times \cR_{\bP, \bN}\times \bN$, and consider the
	pull-backs of $\omega_{\cT_{\bP, \bN}}$ and $\omega_{\cR_{\bP, \bN}} \boxtimes \cF^v_{\bP, \bN}$ separately.
	Let us first consider $\omega_{\cR_{\bP, \bN}}\boxtimes \cF^v_{\bP, \bN}$. We have
	
	\begin{center}
\includegraphics{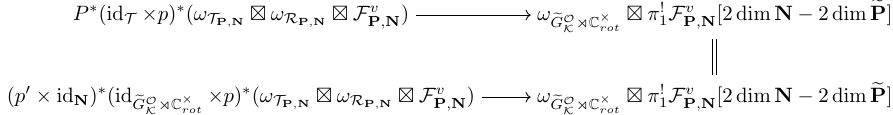}
	\end{center}
	
	by following left, top arrows and bottom, right arrows in the rear square. They are the
	same, as both are essentially given by the homomorphism
	$$p^*\omega_{\cR_{\bP, \bN}}\boxtimes \cF^v_{\bP, \bN} \to \pi_1^!\cF^v_{\bP, \bN}.$$
	Next consider $\omega_{\cT_{\bP, \bN}}$. The $\cT_{\bP, \bN}$-component of
	\(
	(\id_{\cT_{\bP, \bN}}\times p)\circ P =
	(\id_{\tG_{\cK}^{\cO}\rtimes \C^\times_{rot}}\times p)\circ(p'\times\id_{\cR_{\bP, \bN}})
	\)
	(which is $(g_1,g_2, s)\mapsto [g_1, g_1g_2.s]$)
	factors as
	\[
	\tG_{\cK}^{\cO}\rtimes \C^\times_{rot}\times \cP_{\bN}
	\xrightarrow{\id_{\tG_{\cK}^{\cO}\rtimes \C^\times_{rot}}\times\Pi'}
	\tG_{\cK}^{\cO}\rtimes \C^\times_{rot} \times \bN \xrightarrow{p'_{\cT_{\bP, \bN}}}\cT_{\bP, \bN},
	\]
	where $\Pi'\colon \cP_{\bN}\to \bN$ is $(g_2, s)\mapsto g_2.s$. So we have
	\begin{equation*}
	((\id_{\cT_{\bP, \bN}}\times p)\circ P)^*(\omega_{\cT_{\bP, \bN}}\boxtimes\omega_{\cR_{\bP, \bN}}\boxtimes \cF^v_{\bP, \bN})
	\cong \omega_{\tG_{\cK}^{\cO}\rtimes \C^\times_{rot}}\boxtimes \pi_1^!\cF^v_{\bP, \bN}
	[2\dim \bN-2\dim \tbP].
	\end{equation*}
	The two restriction with supports homomorphisms  from above constructed by going along left, top arrows and
	bottom, right arrows in the rear square are thus identical. This completes the proof of the commutativity of the bottom left
	square.
	
	Since $\tilde q\colon \cP_{\bN}\to q(\cP_{\bN})$
	is a fiber bundle with fibers $\tbP\rtimes \C^\times_{rot}$, commutativity for squares
	involving $q$ is obvious. 
	

	Let us finally consider the right bottom
	square. We extend it to a cube:
	\begin{equation*}
	\xymatrix@!C=100pt{
		& \tG_{\cK}^{\cO}\rtimes \C^\times_{rot}\times q(\cP_{\bN}) \ar@{->}'[d]^{P'}[dd]
		\ar@{->}[rr]_{\id_{\tG_{\cK}^{\cO}\rtimes \C^\times_{rot}}\times\tilde m}
		&& \tG_{\cK}^{\cO}\rtimes \C^\times_{rot}\times \bN \ar@{->}[dd]^p
		\\
		\boxed{2} \ar@{->}[ur]\ar@{->}[rr]\ar@{->}[dd] && \cP_{\bN}
		\ar@{->}[dd] \ar@{->}[ur]
		&
		\\
		& \cT_{\bP, \bN}\times q(\cP_{\bN}) \ar@{->}'[r][rr]_{\id_{\cT_{\bP, \bN}}\times\tilde m}
		&& \cT_{\bP, \bN}\times \bN
		\\
		\cR_{\bP, \bN}\times q(\cP_{\bN}) \ar@{->}[rr]\ar@{->}[ur]
		&& \cR_{\bP, \bN}\times \bN\ar@{->}[ur]
	}
	\end{equation*}
	Arrows from the front to rear are closed embeddings. The map $P'\colon  \tG_{\cK}^{\cO}\rtimes \C^\times_{rot}\times q(\cP_{\bN}) \to \cT_{\bP, \bN}\times q(\cP_{\bN})$ is given by 
	$$(g_1, [g_2,s])\mapsto ([g_1, g_1g_2.s], [g_2, s]).$$
	
	The left and right faces of the cube are cartesian, and the commutativity of the rear square in the cube is enough to conclude that  the corresponding proper pushforwards give the same map.
	
	Finally, the commutativity of the induced maps in the right top square is clear, as it involves
	only pushforward homomorphisms. In particular, the whole large square is commutative.
\end{proof}

\begin{lemma}
	\label{lemma:unit}
	The class of $[1]\in H_*^{\tbP\rtimes \C^\times}(\cR_{\bP, \bN})$ acts by the identity on $H^{L_v}_*(M_v^{\bP, \bN})$.
\end{lemma}
\begin{proof}
	Consider the following diagram.
	\begin{center}
\includegraphics{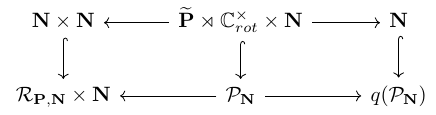}
	\end{center}
	The vertical maps are the natural inclusions (where we include $\bN\hookrightarrow \cR_{\bP, \bN}$ as the fiber over $\Fl_\bP^{\leq 1}$).
	Since $[1]\otimes  c$ is the pushforward of $1\otimes c$ along the left inclusion, by proper base change, $q_*p^*([1]\otimes  c)$ is given by the pushforward along right vertical embedding $$ \bN \to q(\cP_{\bN}).$$ Composing with $m: q(\cP_{\bN})\to \bN$, this embedding becomes the identity map on $ \bN$, so we must have 
	$m_*q_*p^*([1]\otimes  c)=c$.
\end{proof}

\bibliographystyle{plainnat}

\end{document}